\documentclass[12pt,twoside,reqno]{amsart}
\usepackage{bbm}
\usepackage{amsxtra,amscd}
\usepackage{graphicx}
\usepackage{amsmath}
\usepackage{amsfonts}
\usepackage{amssymb}

\newtheorem{theorem}{Theorem}[section]
\newtheorem{lemma}[theorem]{Lemma}
\theoremstyle{definition}
\newtheorem{definition}[theorem]{Definition}

\newtheorem{remark}[theorem]{Remark}

\numberwithin{equation}{section}

\begin{document}
\title{Algebraic Anabelian Functors}
\author{Feng-Wen An}
\address{School of Mathematics and Statistics, Wuhan University, Wuhan,
Hubei 430072, People's Republic of China}
\email{fwan@amss.ac.cn}

\begin{abstract}
In this paper we will prove that there exists  a covariant functor, called algebraic anabelian functor, from the category of algebraic schemes over a given field to the category of outer homomorphism sets of groups. The algebraic anabelian functor, given in a canonical manner, is full and faithful. It  reformulates the anabelian geometry over a field. As an application of the anabelian functor, we will also give a proof of the section conjecture of Grothendieck for the case of algebraic schemes.
\end{abstract}

\maketitle

\begin{center}
{\tiny {Contents} }
\end{center}

{\tiny \qquad {Introduction} }

{\tiny \qquad {1. Statement of the Main Theorem} }

{\tiny \qquad {2. Application to the Section Conjecture} }

{\tiny \qquad {3. Basic Definitions} }

{\tiny \qquad {4. Galois Extensions of Function Fields} }

{\tiny \qquad {5. Key Property of \emph{qc} Schemes} }

{\tiny \qquad {6. Universal Construction for \emph{qc} Schemes} }

{\tiny \qquad {7. Main Property of \emph{qc} Schemes} }

{\tiny \qquad {8. \emph{sp}-Completion} }

{\tiny \qquad {9. Unramified Extensions of Function Fields} }

{\tiny \qquad {10. Algebraic Fundamental Groups} }

{\tiny \qquad {11. Monodromy Actions} }

{\tiny \qquad {12. Proof of the Main Theorem}}

{\tiny \qquad {References}}

\section*{Introduction}

In this paper we will prove that there exists an \emph{algebraic anabelian functor}, a covariant functor, given in a canonical manner, from the category of algebraic schemes over a given field to the category of outer homomorphism sets of groups.

Fortunately, the algebraic anabelian functor is full and faithful. In deed, it reformulates the anabelian geometry over a field in the sense of Grothendieck. For detail, see \emph{Theorem 1.3}, the main theorem of the paper.

As an application of the algebraic anabelian functor, we will give a proof of the section conjecture of Grothendieck for the case of algebraic schemes. See \S \emph{2} for detail.

We will prove the Main Theorem of the paper in \S \emph{12} after we make several preparations in \S\S \emph{3-11}.

In particular, in \S \emph{9} we will give the proofs that the arithmetic unramified extension in \cite{An4*} and the formally unramified extension in \cite{An6} are both well-defined. These unramified extensions are used to give the computations of \'{e}tale fundamental groups for arithmetic schemes and algebraic schemes, respectively.

Note that there exists another anabelian functor, the \textbf{arithmetic anabelian functor}, which is a covariant functor defined canonically on the category of arithmetic schemes surjectively over  the ring $\mathcal{O}_{K}$ of algebraic integers of a number field $K$. Such a functor is also full and faithful.

However, the arithmetic anabelian functors, which are related to class field theory, are very different from the algebraic ones. This is due to the fact that their \'{e}tale fundamental groups are very different. For example, it is clear that $$\pi _{1}^{et}(Spec(\mathbb{Q}))\cong Gal(\overline{\mathbb{Q}}/\mathbb{Q})$$ holds (or see \emph{Theorem 10.2} for a generalized result).
On the other hand, by \cite{An4*,An8} we have $$\pi _{1}^{et}(Spec(\mathbb{Z}))\cong Gal(\mathbb{Q}^{un}/\mathbb{Q})=\{0\}.$$ Here, $\mathbb{Q}^{un}$ ($=\mathbb{Q}$) denotes the (nonabelian) maximal unramified extension of the rational field $\mathbb{Q}$.

\bigskip

\textbf{{\tiny {Acknowledgment.}}} The author would like to express his
sincere gratitude to Professor Li Banghe for his advice and instructions on
algebraic geometry and topology.

\section{Statement of the Main Theorem}

\subsection{Notation}

Let $K$ be a field.
By an \textbf{algebraic $K$-variety} we will understand an integral scheme $X
$ over  $K$ of finite type.

For an integral scheme $Z$, put

\begin{itemize}
\item $k(Z)\triangleq \mathcal{O}_{X,\xi}$, the function field of an integral scheme $Z$ with generic point $\xi$;

\item $\pi _{1}^{et}\left( Z\right) \triangleq$ the \'{e}tale fundamental
group of $Z$ for a geometric point of $Z$ over a separable closure of the
function field $k\left( Z\right).$
\end{itemize}

In particular, for a field $L$, set
$$\pi _{1}^{et}(L)\triangleq\pi _{1}^{et}(Spec(L)).$$

\subsection{Outer homomorphism set}

Let $G,H,\pi_{1},\pi_{2}$ be four groups with homomorphisms $p:G\to \pi_{1}$
and $q:H\to \pi_{2}$, respectively.

\begin{definition}
The \textbf{outer homomorphism set} from $G$ into $H$ over $\pi_{1}$ and $\pi_{2}$ respectively, denoted by $Hom_{\pi_{1},\pi_{2}}^{out}(G,H)$, is the
set of the maps $\sigma$ from the quotient $\frac{\pi_{1}}{p(G)}$ into the
quotient $\frac{\pi_{2}}{q(H)}$, given by a group homomorphism $f:G\to H$ in
such a manner:
\begin{equation*}
\sigma:\frac{\pi_{1}}{p(G)} \to \frac{\pi_{2}}{q(H)}, x\cdot p(G)\mapsto
f(x)\cdot q(H)
\end{equation*}
for any $x\in \pi_{1}$.

In particular, such a $\sigma$ is said to be \textbf{bijective} if the $f$ above is an isomorphism such that $q(H)=f\circ p (G)$.
\end{definition}

\begin{remark}
Suppose that $G$
and $H$ are normal subgroups of $\pi_{1}$ and $\pi_{2}$, respectively. Then $Hom_{\pi_{1},
\pi_{2}}^{out}(G,H)$ can be regarded as a subset of $Hom(Out(G),Out(H))$.
\end{remark}

Let $\mathbb{P}\mathbb{G}$ be the category of group pairs $(G,\pi)$ as objects,  with outer homomorphisms between group pairs as morphisms. Here, By a \textbf{group pair} $(G,\pi)$ we understand two groups $G$ and $\pi$ together with a given homomorphism $p:G\to \pi$ of groups. By an \textbf{outer homomorphism} from $(G,\pi_{1})$ into $(H,\pi_{2})$  we understand a map $$\sigma: \frac{\pi_{1}}{p(G)}\to \frac{\pi_{2}}{q(H)}$$ given in \emph{Definition 1.1}.

$\mathbb{P}\mathbb{G}$ will be called the \textbf{category of outer homomorphism sets of groups} in the paper.

\subsection{Statement of the main theorem}

Fixed a field $K$. Here $K$ is not necessarily of characteristic zero.  For any algebraic $K$-variety $X$, there canonically exists a group pair $({ \pi _{1}^{et}\left( X\right)},{\pi _{1}^{et}(k(X))})$.

Let ${\mathbb{S}\mathbbm{c}\mathbbm{h}(K)}$ denote the category of algebraic $K$-varieties as objects, with scheme morphisms as morphisms satisfying the condition:
\begin{quotation}
 \emph{For any $X,Y \in {\mathbb{S}\mathbbm{c}\mathbbm{h}(K)}$, a scheme morphism $f:X \to Y$ is said to be contained in the category ${\mathbb{S}\mathbbm{c}\mathbbm{h}(K)}$ if  $k(X)$ is  separably generated  over $k(Y)$ canonically.}
\end{quotation}

Here is the main theorem of the present paper.

\begin{theorem}
\emph{\textbf{(Main Theorem)}}
For any field $K$, there is a covariant functor  $\tau$  from category  ${\mathbb{S}\mathbbm{c}\mathbbm{h}(K)}$ to category $\mathbb{P}\mathbb{G}$ given in a canonical manner:
\begin{itemize}
\item An algebraic $K$-variety $X\in {\mathbb{S}\mathbbm{c}\mathbbm{h}(K)}$ is mapped into a group pair $({ \pi _{1}^{et}\left( X\right)},{\pi _{1}^{et}(k(X))})\in \mathbb{P}\mathbb{G}$;

\item A scheme morphism $f\in Hom(X,Y)$ is mapped into an outer homomorphism $\tau(f)\in Hom_{\pi _{1}^{et}(k(X)),\pi _{1}^{et}(k(Y))
}^{out}\left( \pi _{1}^{et}\left( X\right) ,\pi _{1}^{et}\left( Y\right)
\right)$ given  by $f$.
\end{itemize}
Furthermore, $\tau$ is full and faithful.

In particular, an $f\in Hom(X,Y)$ is an isomorphism if and only if $X$ and $Y$ have a common sp-completion and $\tau(f)$ is a bijective outer homomorphism.
\end{theorem}

Here, for  \emph{sp-completion}, see \cite{An8} or see below \S \emph{8.2} in the present paper. Roughly speaking, an \emph{sp}-completion of an integral scheme is such a one that contains all the separably closed points.

We will prove \emph{Theorem 1.3} in \S \emph{12} after we make preparations in \S\S \emph{3-11}.

\begin{remark}
The  functor $\tau(K)$ is said to be the \textbf{anabelian functor} over a field $K$, or \textbf{algebraic anabelian functor}. The Main Theorem  above says that the ananbelian functor over a field reformulates the anabelian geometry in the sense of Grothendieck. In deed, it will also give an answer to the section conjecture of Grothendieck for the case of algebraic schemes (see \S 2).
\end{remark}

\begin{remark}
There exists an \textbf{arithmetic anabelian functor}, which is the anabelian functor over  the ring $\mathcal{O}_{K}$ of algebraic integers of a number field $K$. Such a functor is also full and faithful. However, the arithmetic anabelian functors, which are related to class field theory, are very different from the algebraic ones. This is due to the fact that their \'{e}tale fundamental groups are very different.
\end{remark}

\begin{remark}
Let $X$ and $Y$ be two integral $K$-varieties. It is seen that the group pairs $({ \pi _{1}^{et}\left( X\right)},{\pi _{1}^{et}(k(X))})$ and $({ \pi _{1}^{et}\left( Y\right)},{\pi _{1}^{et}(k(Y))})$ are indeed the {ramified groups} $\pi_{1}^{br}(X)$ and $\pi_{1}^{et}(Y)$, respectively. Hence, the outer homomorphism set $$Hom_{\pi _{1}^{et}(k(X)),\pi _{1}^{et}(k(Y))
}^{out}\left( \pi _{1}^{et}\left( X\right) ,\pi _{1}^{et}\left( Y\right)
\right)$$ is exactly equal to the set $$Hom(\pi_{1}^{br}(X),\pi_{1}^{br}(Y))$$ of homomorphisms between the ramified groups. For ramified groups, see \S \emph{10-11} below in the paper.
\end{remark}

\section{Application to the Section Conjecture}

\subsection{Algebraic anabelian functor: Special case}

Let ${\mathbb{S}\mathbbm{c}\mathbbm{h}(K)}_{0}$ be the category of algebraic $K$-varieties as objects, together with scheme morphisms as morphisms satisfying the condition:
\begin{quotation}
  \emph{For any $X,Y\in {\mathbb{S}\mathbbm{c}\mathbbm{h}(K)}_{0}$, a morphism $f:X \to Y$ of schemes is said to be contained in the category ${\mathbb{S}\mathbbm{c}\mathbbm{h}(K)}_{0}$ if $X$ and $Y$ have a common $sp$-completion and  $k(X)$ is separable over $k(Y)$ canonically.}
\end{quotation}

Here is a result on algebraic anabelian functor on the category ${\mathbb{S}\mathbbm{c}\mathbbm{h}(K)}_{0}$.

\begin{theorem}
For any field $K$, there exists a covariant functor  $\tau_{0}$  from category  ${\mathbb{S}\mathbbm{c}\mathbbm{h}(K)}_{0}$ to category $\mathbb{P}\mathbb{G}$ given in a canonical manner:
\begin{itemize}
\item An algebraic $K$-variety $X\in {\mathbb{S}\mathbbm{c}\mathbbm{h}(K)}_{0}$ is mapped into a group pair $({ \pi _{1}^{et}\left( X\right)},{\pi _{1}^{et}(k(X))})\in \mathbb{P}\mathbb{G}$;

\item A scheme morphism $f\in Hom(X,Y)$ is mapped into an outer homomorphism $\tau_{0}(f)\in Hom_{\pi _{1}^{et}(k(X)),\pi _{1}^{et}(k(Y))
}^{out}\left( \pi _{1}^{et}\left( X\right) ,\pi _{1}^{et}\left( Y\right)
\right)$ given  by $f$.
\end{itemize}
Furthermore, $\tau_{0}$ is  full and faithful.

In particular, an $f\in Hom(X,Y)$ is an isomorphism if and only if the outer homomorphism  $\tau_{0}(f)$ is  bijective.
\end{theorem}

\begin{proof}
It is immediate from \emph{Theorem 1.3} and \emph{Lemma 12.2}.
\end{proof}

\subsection{Section conjecture for algebraic schemes}

Fixed a field $K$. There are the following results on anabelian geometry of algebraic schemes over $K$.

\begin{theorem}
Let $X$ and $Y$ be two algebraic $K$-varieties. Suppose that $k\left( X\right) $ is canonically separable over $k\left( Y\right)$ and $X,Y$ have a common sp-completion. Then there is a bijection
\begin{equation*}
Hom\left( X,Y\right) \cong Hom_{\pi _{1}^{et}(k(X)),\pi _{1}^{et}(k(Y))
}^{out}\left( \pi _{1}^{et}\left( X\right) ,\pi _{1}^{et}\left( Y\right)
\right)
\end{equation*}
between sets.
\end{theorem}

\begin{proof}
It is immediate from \emph{Theorem 2.1}.
\end{proof}

\begin{theorem}
Let $X$ be an algebraic $K$-variety. Suppose that $k(X)$ is separably generated over $K$. Then there is a bijection
\begin{equation*}
\Gamma \left( X/K\right) \cong Hom_{\pi _{1}^{et}(K),\pi _{1}^{et}(k(X))
}^{out}\left( \pi _{1}^{et}(K) ,\pi _{1}^{et}\left( X\right) \right)
\end{equation*}
between sets.
\end{theorem}

\begin{proof}
It is immediate from \emph{Theorem 1.3} and \emph{Lemma 12.4}.
\end{proof}

\subsection{An interpretation given by Galois groups}

For a field $L$, set the following symbols

\begin{itemize}
\item $G(L)\triangleq$ the absolute Galois group $Gal(L^{sep}/L)$;

\item $G(L)^{au}\triangleq$ the Galois group $Gal(L^{au}/L)$ of the maximal
formally unramified extension $L^{au}$ of $L$ (see \emph{Definition 9.7}
below).
\end{itemize}

For \emph{Theorems 2.2-3},
there is the following version of Galois groups of fields.

\begin{theorem}
Let $X$ and $Y$ be two algebraic $K$-varieties. Suppose that $k\left( X\right) $ is canonically separable over $k\left( Y\right)$ and $X,Y$ have a common sp-completion. Then there is a bijection
\begin{equation*}
Hom\left( X,Y\right) \cong Hom_{G(k(X)),G(k(Y)) }^{out}\left( G(k(X))^{au}
,G(k(Y))^{au} \right)
\end{equation*}
between sets.
\end{theorem}

\begin{theorem}
Let $X$ be an algebraic $K$-variety. Suppose that $k(X)$ is separably generated over $K$. Then there is a bijection
\begin{equation*}
\Gamma \left( X/K\right) \cong Hom_{G(K),G(k(X)) }^{out}\left( G(K)^{au}
,G\left( k(X)\right)^{au} \right)
\end{equation*}
between sets.
\end{theorem}

It is seen  that \emph{Theorems 2.4-5} hold from \emph{Theorems 2.2-3} above.

\section{Basic Definitions}

Let's fix notation and terminology in the present paper. They will be used in the following sections.

\subsection{Convention}

For an integral domain $D$, let $Fr(D)$ denote the field of fractions of $D$.

In particular, let $D$
be contained in a field $\Omega $.
In the paper $Fr(D)$ will always be assumed to be contained in $\Omega $.

For a field $L$, set
\begin{itemize}
\item $L^{sep}\triangleq$ the separable closure of $L$;

\item $L^{al}$ (or $\overline{L}$) $\triangleq$ an algebraic closure of  $L$.
\end{itemize}

By an \textbf{integral $K$-variety} $X$ in the paper we will understand an integral scheme over a field $K$ (not necessarily of finite type).

\subsection{Quasi-galois extension of a function field}

Assume that $L $ is an extension over a  field $K$.  Let $Gal(L/K)$ be the Galois group of $L$
over $K$.
Note that here  $L$ is not necessarily algebraic over $K$.

\begin{definition}
The field
$L$ is said to be \textbf{Galois} over $K$ if $K$ is the
invariant subfield of $Gal(L/K)$.
\end{definition}

For example, $\mathbb{Q}(t)$ is Galois over $\mathbb{Q}$. Here, $t$ is a variable over $\mathbb{Q}$.

Now we extend the notion of quasi-galois  from algebraic extensions to function fields.

\begin{definition}
The field
$L$ is said to be \textbf{quasi-galois} over $K$ if each irreducible
polynomial $f(X)\in F[X]$ that has a root in $L$ factors completely in $L%
\left[ X\right] $ into linear factors for any subfield $F$ with $K\subseteq
F\subseteq L$.
\end{definition}

Let $D\subseteq D_{1}\cap D_{2}$ be three integral domains. Then $D_{1}$ is
said to be \textbf{quasi-galois} over $D$ if the fraction field $Fr\left( D_{1}\right) $ is
quasi-galois over $Fr\left( D\right) $.

\begin{definition}
The ring
$D_{1}$ is said to be a \textbf{conjugation} of $D_{2}$ over $D$ if there is
an $F-$isomorphism $\tau:Fr(D_{1})\rightarrow Fr(D_{2})$ such that $\tau(D_{1})=D_{2}$,
where $F\triangleq k(\Delta)$, $k\triangleq Fr(D)$, $\Delta$ is a transcendental basis of the field $Fr(D_{1})$ over $k$, and $F$
is contained in $Fr(D_{1})\cap Fr(D_{2})$.

In such a case, $D_{1}$ is also said to be a \textbf{$D-$conjugation} of $D_{2}$.
\end{definition}

Replacing rings by fields, we have a definition that a field $L_{1}$ is said to be a \textbf{conjugation} of a field $L_{2}$ over a field $K$, where $K$ is assumed to be contained in the intersection $L_{1}\cap L_{2}$. Note that in such a case, we must have $\overline{L_{1}}=\overline{L_{2}}$.

\subsection{Essentially affine scheme}

Let $X$ be a scheme. As usual, an \textbf{affine covering} of $X$
is a family $\mathcal{C}_{X}=\{(U_{\alpha },\phi _{\alpha };A_{\alpha
})\}_{\alpha \in \Delta }$ such that for each $\alpha \in \Delta $, $\phi
_{\alpha }$ is an isomorphism from scheme $(U_{\alpha },\mathcal{O}_{X}|_{U_{\alpha }})$ onto scheme $(Spec{( A_{\alpha })},\mathcal{O}_{Spec{( A_{\alpha }})})$, where $A_{\alpha }$ is
a commutative ring with identity.

Each element $(U_{\alpha },\phi _{\alpha };A_{\alpha })\in \mathcal{C}_{X}$ is
called a \textbf{local chart}. For the sake of brevity, a local chart $%
(U_{\alpha},\phi_{\alpha};A_{\alpha })$ will be denoted by $U_{\alpha}$ or $%
(U_{\alpha},\phi_{\alpha})$.

An affine covering $\mathcal{C}_{X}$ of $(X, \mathcal{O}_{X})$ is said to be
\textbf{reduced} if $U_{\alpha}\neq U_{\beta} $ holds for any $\alpha\neq
\beta$ in $\Delta$.

\begin{definition}
An affine covering $\{(U_{\alpha },\phi _{\alpha };A_{\alpha })\}_{\alpha
\in \Delta }$ of $X$ is said to be an \textbf{affine patching} of $X$ if $U_{\alpha }=SpecA_{\alpha }$ and the
map $\phi _{\alpha }$ is the identity map on the underlying space $U_{\alpha }$
for each $\alpha \in \Delta .$
\end{definition}

Let $\mathfrak{Comm}$ be the category of commutative rings with identity.
For a given field $\Omega$, let $\mathfrak{Comm}(\Omega)$ be the category
consisting of the subrings of $\Omega$ and their isomorphisms.

\begin{definition}
Let $\mathfrak{Comm}_{0}$ be a subcategory of $\mathfrak{Comm}$.
 An affine
covering $\{(U_{\alpha},\phi_{\alpha};A_{\alpha })\}_{\alpha \in \Delta}$ of
$X$ is said to be \textbf{with values} in $\mathfrak{Comm}_{0}$ if for each $%
\alpha \in \Delta$ there are $\mathcal{O}_{X}(U_{\alpha})=A_{\alpha}$ and $U_{\alpha}=Spec(A_{\alpha})$, where
 $A_{\alpha }$ is a ring contained in $\mathfrak{Comm}_{0}$.

In particular, an affine covering $\mathcal{C}_{X}$ of $X$ with values in $%
\mathfrak{Comm}(\Omega)$ is said to be \textbf{with values} in the field $\Omega$.
\end{definition}

Let $\mathcal{C}_{X}=\{(U_{\alpha},\phi_{\alpha};A_{\alpha })\}_{\alpha \in \Delta}$ be a reduced affine covering of
$X$ with values in a field $\Omega$. For any $(U_{\alpha},\phi_{\alpha};A_{\alpha }),(U_{\beta},\phi_{\beta};A_{\beta})\in \mathcal{C}_{X}$, we say $$(U_{\alpha},\phi_{\alpha};A_{\alpha })=(U_{\beta},\phi_{\beta};A_{\beta})$$ if and only if $$U_{\alpha}=U_{\beta}, \, \phi_{\alpha}=\phi_{\beta}.$$ That is, we will always neglect the map $\phi_{\alpha}$ for a local chart $(U_{\alpha},\phi_{\alpha};A_{\alpha })$ in such a $\mathcal{C}_{X}$.

For brevity, a scheme is said to be \textbf{essentially affine} in $\Omega$ if it has a reduced affine covering with values in $\Omega$.

It will be seen that essentially affine schemes have many properties like affine schemes.

\subsection{Essentially equal scheme}

By affine covering with values in a field, it is seen that affine open sets in a scheme is measurable and the non-affine open sets are unmeasurable.
So we can neglect the non-affine open sets in an evident manner, where almost every property of the scheme is preserved.

Now suppose that there are  two structure sheaves $\mathcal{O}_{X}$ and $\mathcal{O}^{\prime}_{X}$  on the underlying space of an integral scheme $X$.

\begin{definition}
 The two integral schemes $(X,\mathcal{O}_{X})$ and $(X, \mathcal{O}^{\prime}_{X})$ are said to be \textbf{essentially equal} provided that for any open set $U$ in $X$, there is an equivalence relation
 $$U \text{ is affine open in }(X,\mathcal{O}_{X}) \Longleftrightarrow \text{ so is }U \text{ in }(X,\mathcal{O}^{\prime}_{X})$$ and in such a case, either $D_{1}=D_{2}$ holds or  the two conditions below are both satisfied\footnote{By such additional conditions, we have a sufficiently large number of integral schemes so that we can give a computation of \'{e}tale fundamental groups.}:
  \begin{itemize}
  \item $Fr(D_{1})=Fr(D_{2})$.

  \item For any nonzero $x\in Fr(D_{1})$, there is a relation $$x\in D_{1}\bigcap D_{2}$$ or there is an equivalence relation $$x\in D_{1}\setminus D_{2} \Longleftrightarrow x^{-1}\in D_{2}\setminus D_{1}.$$
  \end{itemize}
 Here, $D_{1}=\mathcal{O}_{X} (U)$ and $D_{2}=\mathcal{O}^{\prime}_{X} (U)$.
\end{definition}

For example, consider the discrete valuation ring $$\mathbb{Z}_{(p)}=\{\frac{r}{s}:r,s\in \mathbb{Z},s\not=0,(p,s)=1\}$$ of $Spec(\mathbb{Q})$ for a prime $p$. It is clear that $Spec(\mathbb{Z}_{(3)})$ and $Spec(\mathbb{Z}_{(5)})$ are not essentially equal.

In deed, let $\dim X=1$. Suppose that the integral schemes $(X,\mathcal{O}_{X})$ and $(X, \mathcal{O}^{\prime}_{X})$ are {essentially equal}. Then $\mathcal{O}_{X}(U)$ and $\mathcal{O}^{\prime}_{X}(U)$ have the same discrete valuation for an affine open set $U$ in $X$.

\begin{definition}
Any two schemes $(X,\mathcal{O}_{X})$ and $(Z,\mathcal{O}_{Z})$ are said to be
\textbf{essentially equal} if the underlying spaces of $X$ and $Z$ coincide
with each other and the schemes $(X,\mathcal{O}_{X})$ and $(X,\mathcal{O}_{Z})$ are essentially equal.
\end{definition}

It is  seen that scheme that are essentially equal must be isomorphic.

\subsection{Quasi-galois closed affine covering}

Assume that $f:X\rightarrow Y$ is a surjective morphism between
integral schemes. Fixed an algebraic closure $\Omega$ of the function field $%
k(X)$.

\begin{definition}
A reduced affine covering $\mathcal{C}_{X}$ of $X$ with values in $\Omega $
is said to be \textbf{quasi-galois closed} over $Y$ by $f$ if there exists a
local chart $(U_{\alpha }^{\prime },\phi _{\alpha }^{\prime };A_{\alpha
}^{\prime })\in \mathcal{C}_{X}$ such that $U_{\alpha }^{\prime }\subseteq
\varphi^{-1}(V_{\alpha})$ holds

\begin{itemize}
\item for any affine open set $V_{\alpha}$ in $Y$;

\item for any $(U_{\alpha },\phi _{\alpha };A_{\alpha })\in \mathcal{C}_{X}$ with $U_{\alpha }\subseteq
f^{-1}(V_{\alpha})$;

\item for any conjugate $A_{\alpha }^{\prime }$ of $A_{\alpha }$ over $%
B_{\alpha}$,
\end{itemize}
where $B_{\alpha}$ is the canonical image of $\mathcal{O}_{ Y}(V_{\alpha})$
in $k(X)$ via $f$.
\end{definition}

\subsection{Quasi-galois closed scheme}

Let $X$ and $Y$ be integral schemes. Suppose that
 $f:X\rightarrow Y$ is a surjective morphism. Denote by $Aut\left( X/Y\right) $ the group of
automorphisms of $X$ over $Y$.

An integral scheme $Z$ is said to be a \textbf{conjugate} of $X$ over $Y$ if
there is an isomorphism $\sigma :X\rightarrow Z$ over $Y$.

\begin{definition}
The scheme $X$ is said to be \textbf{quasi-galois closed} (or \textbf{\emph{qc}} for short)
over $Y$ by $f$ if there is
an algebraically closed field $\Omega$ and a reduced affine covering $\mathcal{C}_{X}$ of $X$ with values in $\Omega $ such that for any conjugate
$Z$ of $X$ over $Y$ the two conditions are both satisfied:

\begin{itemize}
\item $(X,\mathcal{O}_{X})$ and $(Z,\mathcal{O}_{Z})$ are essentially equal if $Z$ is essentially affine in $\Omega$.

\item $\mathcal{C}_{Z}\subseteq \mathcal{C}_{X}$ holds if $\mathcal{C}_{Z}$
is a reduced affine covering of $Z$ with values in $\Omega $.
\end{itemize}
\end{definition}

\begin{remark}
In the above definition, the field $\Omega$ enables the affine
open subschemes of the integral scheme $X$ to be \emph{measurable} while the other open
subschemes of $X$ are still \emph{unmeasurable}. In particular, each scheme $X$ that is \emph{qc} over $Y$ must be essentially affine.
\end{remark}

\begin{remark}
Let $\Omega$ be the algebraically closed field  in \emph{Definition 3.9}.

$(i)$ By  $\Omega$, all the rings of affine open sets in
$X$ are taken to be as subrings of the same ring $\Omega$ so that they can
be compared with each other.

$(ii)$ By $\Omega$, we can restrict ourselves only to consider
the function fields which have the same variables over a given field.
\end{remark}

\begin{remark}
It is seen that in \emph{Definition 3.9}, the affine covering $\mathcal{C}_{X}$ of $X$ is
maximal by set inclusion. In fact, $\mathcal{C}_{X}$ is the \emph{natural affine structure}
of $X$ with values in $\Omega$ (see \cite{An2} for definition). Conversely, it can be proved
that a quasi-galois closed scheme has a unique natural affine structure
with values in $\Omega$ (see \cite{An2,An7}).

In other words, $\Omega$ can be chosen to be an algebraic closure of the
function field $k(X)$; $\mathcal{C}_{X}$ is the unique maximal affine
covering of $X$ with values in $\Omega$ (see \emph{Remark 5.2} in \S \emph{5} below).
\end{remark}

\section{Galois Extensions of Function Fields}

\subsection{Galois extension of a function field}

In this section we will prove the following result about a Galois extension of a function field.

\begin{theorem}
Let $L$ be a finitely generated extension of a field $K$. Then  $L$ is Galois over $K$ if and only if  $L$ is quasi-galois  and separably generated over $K$.
\end{theorem}

Note that here $L$ is not necessarily algebraic over $K$. After several lemmas, we will prove \emph{Theorem 4.1} at the end of the present section.

\subsection{Recalling preliminary facts on quasi-galois extensions}

Let $L$ be a finitely generated extension of a field $K$.

The elements $w_{1},w_{2},\cdots ,w_{n}\in L $ are said to be a \textbf{$(r,n)-$nice basis} of $L$ over $K$ if the conditions below
are satisfied:

\begin{itemize}
\item $L=K(w_{1},w_{2},\cdots ,w_{n})$;

\item $w_{1},w_{2},\cdots ,w_{r}$ are a transcendental basis of $L$ over $K$;

\item $w_{r+1},w_{r+2},\cdots ,w_{n}$ are a linear basis of $L$ over $K(w_{1},w_{2},\cdots ,w_{r})$.
\end{itemize}
Here $0\leq r\leq n$.

Let's recall preliminary facts on quasi-galois extensions of function fields.

\begin{lemma}
\emph{(\cite{An2,An2*})}
Fixed an intermediate field $K\subseteq F \subsetneqq L$ and an  $x \in L$ that is algebraic over $F$. Let
$z$ be a conjugate of $x$ over $F$.
Then there exists a $(s,m)-$nice basis $v_{1},v_{2},\cdots ,v_{m}$ of $L$ over $F\left( x\right)$ and an $F-$isomorphism $\tau$ from the field
\begin{equation*}
L=F\left( x,v_{1},v_{2},\cdots ,v_{s},v_{s+1},\cdots, v_{m}\right)
\end{equation*}
onto a field of the form
\begin{equation*}
F\left( z,v_{1},v_{2},\cdots ,v_{s},w_{s+1},\cdots, w_{m}\right)
\end{equation*}
such that
\begin{equation*}
\tau (x)=z,\tau (v_{1})=v_{1},\cdots,\tau (v_{s})=v_{s}.
\end{equation*}
Here, $w_{s+1},w_{s+2},\cdots, w_{m}$ are elements contained in an extension
of  $F$.

In particular, we have
\begin{equation*}
w_{s+1}=v_{s+1},w_{s+2}=v_{s+2},\cdots, w_{m}=v_{m}
\end{equation*}
if $z$ is not contained in $F(v_{1},v_{2},\cdots,v_{m})$.
\end{lemma}

\begin{proof}
It is immediate from preliminary facts on fields.
\end{proof}

\begin{lemma}
\emph{(\cite{An2,An2*})}
The following
statements are equivalent.

$(i)$ $L$ is quasi-galois over $K$.

$(ii)$ Any conjugation of $L$ over $K$ is contained in $L$.

$(iii)$ There exists one and only one conjugation of $L$
over $K$.

$(iv)$ Take any $x\in L$ and any subfield $K\subseteq
F\subseteq L$. Then $L$ contains all conjugations of $F\left( x\right) $ over $F$.
\end{lemma}

\begin{proof}
Prove $(i)\implies (iv)$. Fixed an $x\in L$ and a
subfield $K\subseteq F\subseteq L.$ If $x$ is a
variable over $F$, the field $F\left( x\right) $ must be contained in $L$ since $F\left( x\right) $
 is the unique conjugation of $F\left( x\right) $ over $F$ by $(i)$.

Let $x$ be algebraic over $F$. Then an $F-$conjugation of $F\left(
x\right)$, which is exactly an $F-$conjugate of $F\left( x\right)$, must be
contained in $L$ by $(i)$.

Prove $(iv) \implies (i)$. Fixed any
subfield $K\subseteq F\subseteq L$. Let $f(X)$ be an  irreducible polynomial over $F$. Take any $x\in L$ such that $f\left(
x\right) =0$. It is seen that an $F-$conjugation of $F(x)$ is nothing other than an $F-$conjugate. Then every $F-$conjugate of $F(x) $ is contained
in $L$ by $(iv)$. Hence, $L$ is quasi-galois  over $K$.

Prove $(iv) \implies (ii)$. Let $H$ be a conjugation of $L$ over $K$.
Fixed any $x_{0}\in H$. Take a $(r,n)-$nice basis $w_{1},w_{2},\cdots
,w_{n} $ of $L$ over $K$ and an isomorphism $\sigma:H\rightarrow L$ over
\begin{equation*}
K_{0}\triangleq K(w_{1},w_{2},\cdots ,w_{r} )
\end{equation*}
such that  $H$ is a $K$-conjugation of $L$ by $\sigma$.

It is clear that $w_{1},
w_{2},\cdots ,w_{r} $ are all contained in the intersection of $H$ and
$L$. It follows that $x_{0}$ must be algebraic over $K_{0}$.

Evidently, the field $K_{0}[x_{0}]$ is a conjugate of the field $K_{0}[\sigma(x_{0})]$ over  $K_{0}
$ and then is a conjugation of $K_{0}[\sigma(x_{0})]$ over  $K$. From $(iv)$ we have $x_{0}\in K_{0}[x_{0}]\subseteq L$. Hence, $H
\subseteq L$.

Prove $(ii)\implies (iv)$. Take any $x\in L$ and any
 subfield  $K\subseteq F\subseteq L$.

If $x$ is a variable over $F$, the field $F\left( x\right) $ that is the unique conjugation of $F\left( x\right) $ itself over $F$ must be
contained in $L$ by $(ii)$.

Suppose that $x$ is algebraic over $F$. Let $z$ be an $F-$conjugate of $x
$. If $F=L$, we have $z=x\in L$ by $(ii)$.

Now let $F\not=L$. From \emph{Lemma 4.2} we have a field of the form
\begin{equation*}
F\left( z,v_{1},v_{2},\cdots ,v_{s},w_{s+1},\cdots, w_{m}\right),
\end{equation*}
which is an $F-$conjugation of $L$. As $K \subseteq F$, we must have
\begin{equation*}
z\in F\left( z,v_{1},v_{2},\cdots ,v_{s},w_{s+1},\cdots,
w_{m}\right)\subseteq L.
\end{equation*}
by $(ii)$ again. Hence, $z \in L$.

Prove $(iii) \implies (ii)$. Trivial.

Prove $(i) \implies (iii)$. Let $L$ be quasi-galois over $K$ and let $H$ be a
conjugation of $L$ over $K$. In the following we will prove $H=L$.

In fact, choose a $(s,m)-$nice basis $v_{1},v_{2},\cdots ,v_{m}$ of $L$ over $K$ and  an $F-$isomorphism $\tau$ of $H$ onto $L$ such that via
 $\tau$ the field $H$ is a conjugate of $L$ over $F$, where
$$
F\triangleq k(v_{1},v_{2},\cdots ,v_{s}).
$$
It is seen that
$
F\subseteq H\subseteq L
$ hold by \emph{Definition 3.2}.

Hypothesize $H\subsetneqq L$. Take any $x_{0}\in L\setminus H$. There are two cases.

\emph{Case (i)}. Let $x_{0}$ be a variable over $H$. We have
\begin{equation*}
\dim_{K}H=\dim_{K}L=s< \infty
\end{equation*}
since $H$ and $L$ are conjugations over $K$. On
the other hand, we have
\begin{equation*}
1+\dim_{K}H=\dim_{K}H(x_{0})\leq\dim_{K}L
\end{equation*}
from $x_{0}\in L\setminus H$, which  will be in contradiction.

\emph{Case (ii)}. Let $x_{0}$ be algebraic over $H$. It is seen that $x_{0}$ is algebraic over $F$. We have
\begin{equation*}
[H:F]=[L:F]< \infty
\end{equation*}
since $H$ is a conjugate of $L$ over $F$. On the other hand, we have
\begin{equation*}
2+[H:F]\leq[H[x_{0}]:F]\leq[L:F]
\end{equation*}
from $x_{0}\in L\setminus H$, which  will be in contradiction.

Then $L\setminus H$ must be empty. Hence,  $L=H$.

This completes the proof.
\end{proof}

\subsection{Proof of \emph{Theorem 4.1}}

Now we give the proof of \emph{Theorem 4.1}:

\begin{proof}

Prove $\Longleftarrow$. Let $L$ be quasi-galois and separably generated over $K$.
If $L$ is algebraic over $K$, it is clear that $L$ is Galois over $K$.

Now suppose that $L$ is a transcendental extension over $K$. It suffices to prove that there exists an automorphism
$
\sigma_{0} \in Gal(L/K)
$
such that $K$ is the invariant subfield of $\sigma_{0}$.

In fact, fixed any $(r,n)-$nice basis $v_{1},v_{2},\cdots ,v_{n}$ of $L$ over $K$. Put $$F_{0}\triangleq K(v_{1},v_{2},\cdots ,v_{r}).$$

Then $L$ is algebraic over $F_{0}$. Here, we have $r\geqslant 1$. By \emph{Lemma 4.3} it is seen that every
conjugation of $L$ over $K$ is exactly $L$ itself. It follows that there is one and only conjugate of $L$ over $F_{0}$. Then $L$ is a quasi-galois algebraic extension of $F_{0}$.

Hence, $L$ is Galois over $F_{0}$ since $L$ is separable over $F_{0}$ from the assumption.
Fixed any $\tau_{0}\in Gal(L/F_{0})$ with $\tau_{0}\not=
id_{L}$.

Let $
\tau_{1}$ be an automorphism of $F_{0}$ over $K$ given by
$$
v_{1}\mapsto \frac{1}{v_{1}}, v_{2}\mapsto \frac{1}{v_{2}}
,\cdots,v_{r}\mapsto \frac{1}{v_{r}}.
$$

Then we have an automorphism $\sigma_{0} \in Gal(L/K)$ defined
by $\tau_{0}$ and $\tau_{1}$ in such a manner
\begin{equation*}
\frac{f(v_{1},v_{2},\cdots ,v_{n})}{g(v_{1},v_{2},\cdots ,v_{n})}\in L
\end{equation*}
\begin{equation*}
\mapsto \frac{f(\tau_{1}(v_{1}),\tau_{1}(v_{2}),\cdots
,\tau_{1}(v_{r}),\tau_{0}(v_{r+1}),\cdots, \tau_{0}(v_{n}))}{
g(\tau_{1}(v_{1}),\tau_{1}(v_{2}),\cdots
,\tau_{1}(v_{r}),\tau_{0}(v_{r+1}),\cdots, \tau_{0}(v_{n}))}\in L
\end{equation*}
for any polynomials $f(X_{1},X_{2},\cdots, X_{n})$ and $g(X_{1},X_{2},%
\cdots, X_{n})\not= 0$ over the field $K$ with $g(v_{1},v_{2},\cdots ,v_{n})\not= 0$.

It is easily seen that we have
\begin{equation*}
g(v_{1},v_{2},\cdots ,v_{n})= 0
\end{equation*}
if and only if
\begin{equation*}
g(v_{1},v_{2},\cdots ,v_{r},\tau_{0}(v_{r+1}),\cdots, \tau_{0}(v_{n}))=0
\end{equation*}
if and only if
\begin{equation*}
g(\tau_{1}(v_{1}),\tau_{1}(v_{2}),\cdots
,\tau_{1}(v_{r}),\tau_{0}(v_{r+1}),\cdots, \tau_{0}(v_{n}))=0.
\end{equation*} Hence, $\sigma_{0}$ is well-defined.

It is seen that $K$ is the invariant subfield of the automorphism $\sigma_{0}$ of $L$ over $K$. Hence,
 $K$ is the invariant subfield of the Galois group $Gal(L/K)$. This proves that $L$ is
Galois over $K$.

Prove $\Longrightarrow$. Let $L$ is Galois over $K$. Fixed any $(r,n)-$nice basis of $L$ over $K$, namely $v_{1},v_{2},\cdots ,v_{n}$. Set $$F_{0}\triangleq K(v_{1},v_{2},\cdots ,v_{r}).$$

By \emph{Lemma 3.1} it is seen that the field $F_{0}$ must be invariant under the Galois group $Gal(L/F_{0})$. It follows that $L$ is a Galois algebraic extension over $F_{0}$. Hence, $L$ is separably generated over $K$.

Let $H$ be a conjugate of $L$ over $F_{0}$. Then $H$ is a conjugation of $L$ over $K$. As $L$ is Galois over $F_{0}$, we have $H=L$; hence, there exists one and only one conjugation of $L$ over $K$ under a $(r,n)-$nice basis of $L$ over $K$.

Now let $$v_{1},v_{2},\cdots ,v_{n}$$ run through all possible $(r,n)-$nice bases of $L$ over $K$. It is seen that there exists one and only one conjugation of $L$ over $K$. From \emph{Lemma 4.3} it is immediate that $L$ is quasi-galois over $K$.
\end{proof}

\section{Key Property of \emph{qc} Scheme}

In this section we will prove a key property of a \emph{qc} scheme and then give a criterion for such a scheme.  By these results, we will obtain an essential and sufficient condition for a \emph{qc} scheme.

\subsection{Key property of a \emph{qc} scheme}

Let $X$ and $Y$ be integral schemes and let $f:X\rightarrow Y$ be a
surjective morphism. We have the following key property for \emph{qc} schemes.

\begin{lemma}
Let $\Omega$ be an algebraic closure of the function field $k(X)$. Suppose that $X$ is {qc}  over $Y$ by $f$.  Then there is a unique
maximal affine covering $\mathcal{C}_{X}$ of $X$ with values in $\Omega $; moreover,
 $\mathcal{C}_{X}$ is quasi-galois closed over $Y$ by $f$.
\end{lemma}

\begin{proof}
From \emph{Definition 3.9} we have
an algebraically closed field $\Omega^{\prime}$ and a reduced affine covering $\mathcal{C}^{\prime}_{X}$ of $X$ with values in $\Omega^{\prime} $ such that for any conjugate
$Z$ of $X$ over $Y$ the two conditions are satisfied:

\begin{itemize}
\item $(X,\mathcal{O}_{X})$ and $(Z,\mathcal{O}_{Z})$ are essentially equal if $Z$ has a reduced affine covering
with values in $\Omega^{\prime}$.

\item $\mathcal{C}_{Z}\subseteq \mathcal{C}^{\prime}_{X}$ holds if $\mathcal{C}_{Z}$
is a reduced affine covering of $Z$ with values in $\Omega ^{\prime}$.
\end{itemize}

It is clear that $\Omega^{\prime}$ contains the function field $k(X)$ and $\mathcal{C}^{\prime}_{X}$ is maximal by set inclusion.

Prove the uniqueness of $\mathcal{C}^{\prime}_{X}$. In deed, let $\mathcal{C}^{\prime\prime}_{X}$ be another reduced affine covering of $X$ with values in $\Omega^{\prime}$ satisfying the two conditions above. Then we must have $\mathcal{C}^{\prime}_{X}=\mathcal{C}^{\prime\prime}_{X}$ according to the second condition.

Prove that $\mathcal{C}^{\prime}_{X}$ is quasi-galois closed over $Y$ by $f$. In fact, take
\begin{itemize}
\item an affine open set $V_{\alpha}$ in $Y$;

\item a $(U_{\alpha },\phi _{\alpha };A_{\alpha })\in \mathcal{C}^{\prime}_{X}$ with $U_{\alpha }\subseteq
f^{-1}(V_{\alpha})$;

\item a conjugate $A_{\alpha }^{\prime }$ of $A_{\alpha }$ over $B_{\alpha}$,
\end{itemize}
where $B_{\alpha}$ is the canonical image of $\mathcal{O}_{ Y}(V_{\alpha})$
in $k(X)$ via $f$.

Then there must exist a
local chart $(U_{\alpha }^{\prime },\phi _{\alpha }^{\prime };A_{\alpha
}^{\prime })\in \mathcal{C}^{\prime}_{X}$ such that $U_{\alpha }^{\prime }\subseteq
\varphi^{-1}(V_{\alpha})$ holds; otherwise, if there is some $(U_{\alpha }^{\prime },\phi _{\alpha }^{\prime };A_{\alpha
}^{\prime })\not\in \mathcal{C}^{\prime}_{X}$, we will obtain a new reduced affine covering $$\{(U_{\alpha }^{\prime },\phi _{\alpha }^{\prime };A_{\alpha
}^{\prime })\}\bigcup \mathcal{C}^{\prime}_{X}$$ of $X$, which is in contradiction to the uniqueness of $\mathcal{C}^{\prime}_{X}$.

It follows that $\mathcal{C}^{\prime}_{X}$ is with values in an algebraic closure $\Omega$ of the function field $k(X)$ from \emph{Definition 3.3}.
\end{proof}

By \emph{Lemma 5.1} we have the following remarks.

\begin{remark}
Let  $\mathcal{C}_{X}$ and $\Omega$ be assumed as in
\emph{Definition 3.9}. There are the following statements.
\begin{itemize}
\item The field $\Omega$ can be chosen to be an algebraic closure of the function field $k(X)$.

\item The affine covering $\mathcal{C}_{X}$ is quasi-galois closed over $Y$ by $f$ and is the unique maximal affine covering of $X$ with values in $\Omega$.
\end{itemize}
\end{remark}

\begin{remark}
The affine covering $\mathcal{C}_{X}$ above is the unique maximal
affine structure of $X$ with values in the algebraic closure $\Omega$ of $k(X)$. In \cite{An2}, we use affine structures to get the key property as in \emph{Lemma 5.1}.
\end{remark}

\subsection{Criterion for \emph{qc} schemes}

Fixed integral schemes $X$ and $Y$. Let $\Omega $ be an
algebraically closed closure of the function field $k\left( X\right) $.

\begin{lemma}
Let $f :X\rightarrow Y$ be a surjective morphism of schemes.
Then  $X$ is qc over $Y$ by $f$ if there is a unique maximal
reduced affine covering $\mathcal{C}_{X}$ of $X$ with values in $\Omega $
such that $\mathcal{C}_{X}$ is quasi-galois closed over $Y$ by $f$.
\end{lemma}

\begin{proof}
Assume that $X$ has a unique maximal reduced affine covering $\mathcal{C}_{X}$ with values in $\Omega $ that is
quasi-galois closed over $Y$ by $f$.

Fixed any conjugate $Z$ of $X$ over $Y$ and any isomorphism $\sigma:Z \rightarrow X$  over $Y$. Suppose that $Z$ has a reduced affine
covering $\mathcal{C}_{Z}$ with values in $\Omega$.

Take a local chart $(W,\delta,C)\in \mathcal{C}_{Z}$. Put
$$
U=\sigma (W);
\,
A=\mathcal{O}_{X}(U);
\,
C=\mathcal{O}_{Z}(W).
$$
We have
\begin{equation*}
U=Spec(A); \, W=Spec(C); \, A\subseteq\Omega; \, C\subseteq\Omega.
\end{equation*}

It is seen that there exists
 an affine open subset $U^{\prime}$ in $X$ such that
\begin{equation*}
C=\mathcal{O}_{X}(U^{\prime})
\end{equation*}
by the assumption that $\mathcal{C}_{X}$ is quasi-galois closed over $Y$.
As
\begin{equation*}
U^{\prime}=Spec(C)=W,
\end{equation*}
we have
\begin{equation*}
W=\sigma^{-1}(U)=U^{\prime}\subseteq X.
\end{equation*}

This proves $Z\subseteq X.$

On the other hand, all such local charts $(W,\delta,C)$ with $Spec(C)=W$ constitute a reduced affine covering $\mathcal{C}^{\prime}_{Z}$ of $Z$ with values in $\Omega $ that is unique, maximal, and
quasi-galois closed over $Y$ by $f\circ \delta$.

In a similar manner, we have $X\subseteq Z.$

Hence, $Z=X.$
It is seen that  $(X,\mathcal{O}_{X})$ and $(Z,\mathcal{O}_{Z})$ are essentially equal.
\end{proof}

\begin{lemma}
\emph{(c.f. \cite{An8})}
Let $f :X\rightarrow Y$ be a surjective morphism.
Then $X$ is
qc over $Y$ if there is a unique maximal affine patching $\mathcal{C}_{X}$ of $X$ with values in $\Omega $ satisfying the condition:
\begin{quotation}
$A_{\alpha }$ has one and only one conjugate over $B_{\alpha}$ for any  $
(U_{\alpha },\phi _{\alpha };A_{\alpha })\in \mathcal{C}_{X}$ and for any
affine open set $V_{\alpha}$ in $Y$ with $U_{\alpha }\subseteq
f^{-1}(V_{\alpha})$, where $B_{\alpha}$ is the canonical image of $
\mathcal{O}_{ Y}(V_{\alpha})$ in $k(X)$.
\end{quotation}
\end{lemma}

\begin{proof}
It is immediate from \emph{Lemma 5.4}.
\end{proof}

\subsection{Equivalent condition}

Now we give an essential and sufficient condition for \emph{qc} schemes.

\begin{theorem}
Let $f :X\rightarrow Y$ be a surjective morphism of schemes.
The following statements are equivalent:
\begin{itemize}
\item The scheme $X$ is
qc over $Y$ by $f$.

\item There is a unique maximal affine
covering $\mathcal{C}_{X}$ of $X$ with values in $\Omega $ such that $\mathcal{C}_{X}$ is quasi-galois closed over $Y$ by $f$.

\item There is a unique maximal affine
patching $\mathcal{C}_{X}$ of $X$ with values in $\Omega $ such that $\mathcal{C}_{X}$ is quasi-galois closed over $Y$ by $f$.
\end{itemize}
\end{theorem}

\begin{proof}
It is immediate from \emph{Lemmas 5.1,5.4-5} above and \emph{Lemma 5.7} below.
\end{proof}

We also need the following lemma to prove the third statement in \emph{Theorem 5.6}.

\begin{lemma}
Let $Z$ be an integral scheme and let $\mathcal{C}_{X}=\{(U_{\alpha },\phi _{\alpha };A_{\alpha
})\}$ be an affine covering of $X$ with values in $\Omega $. Then $\mathcal{C}_{X}$ induces an affine patching $\mathcal{D}_{X}=\{(U_{\alpha },id_{\alpha };\sigma_{\alpha}(A_{\alpha}))\}$ of $X$ with values in $\Omega $ given in a natural manner: $$(U_{\alpha },\phi _{\alpha };A_{\alpha
})\in \mathcal{C}_{X}\mapsto (U_{\alpha },id_{\alpha };\sigma_{\alpha}(A_{\alpha})) \in \mathcal{D}_{X}$$
where $$\sigma_{\alpha}:A_{\alpha}\to A_{\alpha}$$ is a ring isomorphism induced from the homeomorphism $$\phi ^{-1}_{\alpha }:Spec(A_{\alpha})\to U_{\alpha}.$$
\end{lemma}

\begin{proof}
It is immediate from \emph{Definitions 3.4-5}.
\end{proof}

\section{Universal Construction for \emph{qc} Schemes}

In this section we will give a universal construction for a \emph{qc} scheme over a given integral scheme. That is, we will prove the existence of a \emph{qc} cover of an integral scheme.

\subsection{A preliminary lemma}

Let $X$ be an integral scheme that is not essentially affine in $k(X)$. By the lemma below we can change $X$ into a scheme $Z$ that is essentially affine in $k(Z)$. That is, any integral scheme is isomorphic to an essentially affine scheme.

\begin{lemma}
For any integral scheme $X$,
there is an integral scheme $Z$ satisfying the properties:

\begin{itemize}
\item $k\left( X\right) =k\left( Z\right);$

\item $X\cong Z$ are isomorphic schemes;

\item $Z$ is essentially affine in the field $k(Z)^{al}$.
\end{itemize}
\end{lemma}

\begin{proof}
Let $\Omega=k(X)^{al}$ and let $\xi$ be the generic point of $X$. We have $k\left(X\right)\triangleq\mathcal{O}_{X,\xi }$.
For any open set $U$ of $X$, there is
the following canonical embedding
$$i_{U}:\mathcal{O}_{X}(U)\rightarrow k(X).$$

Now take an affine covering $\mathcal{C}_{X}=\{(U_{\alpha },\phi _{\alpha };A_{\alpha
})\}_{\alpha \in \Delta }$ of $X$.  Fixed an $\alpha \in \Delta$. We have ring isomorphisms
$$\phi ^{\sharp}_{\alpha }:A_{\alpha}\to \mathcal{O}_{X}(U_{\alpha});$$
$$i_{U_{\alpha}}:\mathcal{O}_{X}(U_{\alpha})\to B_{\alpha}\subseteq k(X).$$

The isomorphism $$t_{\alpha}\triangleq i_{U_{\alpha}}\circ \phi ^{\sharp}_{\alpha }:A_{\alpha}\to B_{\alpha}$$ between rings induces an isomorphism $$\tau_{\alpha}:(Spec{( A_{\alpha })},\mathcal{O}_{Spec{( A_{\alpha }})})\to (Spec{( B_{\alpha })},\mathcal{O}_{Spec{( B_{\alpha }})})$$ between schemes.

Then we  obtain a new scheme $(X,\mathcal{O}^{\prime}_{X})$, namely $Z$, by gluing these schemes $(U_{\alpha},{\mathcal{O}_{X}|}_{U_{\alpha}})$ along $$\psi _{\alpha }\triangleq \tau _{\alpha }\circ \phi _{\alpha }$$ for $\alpha \in \Delta$.

It is seen that $Z$ has the desired properties.
\end{proof}

\subsection{A universal construction for \emph{qc} covers}

Fixed a field $K$ (not necessarily of characteristic zero). Let $Y$ be an integral $K$-variety.  Suppose that $Y$ is essentially affine in an algebraic closure of $M\triangleq k(Y)$.

Fixed an extension $L$ of $M$ such that $L$ is Galois over $M$. Note that here $L$ is not necessarily finitely generated over $M$.

\begin{lemma}
There exists an integral $K$-variety $X$ and
a surjective morphism $f:X\rightarrow Y$ such that

\begin{itemize}
\item $L=k\left( X\right) $;

\item $f$ is affine;

\item $X$ is qc over $Y$ by $f$;

\item $X$ is essentially affine in $L^{al}$.
\end{itemize}
\end{lemma}

\begin{proof}
(\textbf{Universal Construction for \emph{qc} Covers}) Here, we repeat this construction developed in \cite{An3}. We will  proceed in several steps.

\emph{\textbf{Step 1.}} Take algebraic closures $\Omega _{M}$ of $M$ and $\Omega _{L}$ of $L$, respectively. Suppose $\Omega _{M}\subseteq \Omega _{L}$.

Let $\Delta_{1}$ be a transcendental basis of $L$ over $M$ and let $\Delta_{2}$ be a linear basis of $L$ as a vector space over $M(\Delta_{1})$. Put $$\Delta=\Delta_{1}\bigcup \Delta_{2}.$$

Fixed a reduced affine coverings $\mathcal{C}_{Y}$ of $Y$ with values in $\Omega _{M}$ from the assumption that $Y$ is essentially affine in $\Omega _{M}$. Suppose that $\mathcal{C}_{Y}$ is maximal (by set inclusion).

\emph{\textbf{Step 2.}} Take a local chart $\left( V,\psi
_{V},B_{V}\right) \in \mathcal{C}_{Y}$. It is seen that $V$ is an affine open subset of
$Y $ and we have
$$
Fr\left( B_{V}\right) =M; \, \mathcal{O}_{Y}\left( V\right)
=B_{V}\subseteq \Omega _{M}.
$$

Define $$A_{V}\triangleq B_{V}\left[ \Delta _{V}\right],$$ i.e.,  the subring of $L$ generated over $B_{V}$ by the set
$$
\Delta _{V}\triangleq \{\sigma \left(w\right) \in L:\sigma \in
Gal\left( L/M\right) ,w\in \Delta \}.
$$

Then
$
Fr\left( A_{V}\right) =L
$
holds. It is seen that $B_{V}$ is
exactly the invariant subring of the natural action of the Galois group $Gal\left( L/M\right) $ on $A_{V}.$

Set
\begin{equation*}
i_{V}:B_{V}\rightarrow A_{V}
\end{equation*}
to be the inclusion.

\emph{\textbf{Step 3.}} Define the disjoint union
\begin{equation*}
\Sigma =\coprod\limits_{\left( V,\psi _{V},B_{V}\right) \in \mathcal{C}_{Y}}Spec\left( A_{V}\right) .
\end{equation*}

Then $\Sigma $ is a topological space, where the topology $\tau _{\Sigma }$ on
$\Sigma $ is naturally determined by the Zariski topologies on all $Spec\left( A_{V}\right) .$

Let
\begin{equation*}
\pi _{Y}:\Sigma \rightarrow Y
\end{equation*}
be the projection.

\emph{\textbf{Step 4.}} Given an equivalence relation $R_{\Sigma}$ in $\Sigma$ in such a manner:
\begin{quotation}
\emph{For any $x_{1},x_{2}\in \Sigma $, we say
\begin{equation*}
x_{1}\sim x_{2}
\end{equation*}
if and only if
\begin{equation*}
j_{x_{1}}=j_{x_{2}}
\end{equation*}
holds in $L$.}
\end{quotation}
Here, $j_{x}$ denotes the corresponding prime ideal of $A_{V}$ to a point $%
x\in Spec\left( A_{V}\right) $ (see \cite{EGA}).

Let
\begin{equation*}
X=\Sigma /\sim
\end{equation*}
and let
\begin{equation*}
\pi _{X}:\Sigma \rightarrow X
\end{equation*}
be the projection.

It is seen that $X$ is a topological space as a quotient of $\Sigma .$

\emph{\textbf{Step 5.}} Define a map
\begin{equation*}
f:X\rightarrow Y
\end{equation*}
by
\begin{equation*}
\pi _{X}\left( z\right) \longmapsto \pi _{Y}\left( z\right)
\end{equation*}
for each $z\in \Sigma $.

\emph{\textbf{Step 6.}} Define
$$
\mathcal{C} _{X}=\{\left( U_{V},\varphi _{V},A_{V}\right) \}_{\left( V,\psi
_{V},B_{V}\right) \in \mathcal{C}_{Y}}
$$
where  $U_{V}\triangleq\pi _{Y}^{-1}\left( V\right) $ is an open set in $X$ and $\varphi
_{V}:U_{V}\rightarrow Spec(A_{V})$ is the identity map on $U_{V}$ for each $\left(
V,\psi _{V},B_{V}\right) \in \mathcal{C}_{Y}$.

Then we have an integral scheme
$
\left( X,\mathcal{O}_{X}\right)
$
by gluing the affine schemes $Spec\left( A_{V}\right) $ for
all local charts $\left( V,\psi _{V},B_{V}\right) \in \mathcal{C}_{Y}$ with
respect to the equivalence relation $R_{\Sigma}$ (see \cite{EGA,Hrtsh}).

It is seen that $\mathcal{C} _{X}$ is an
affine patching on the scheme $X$ with values in $\Omega _{L}$.

In particular, $\mathcal{C} _{X}$ is maximal and quasi-galois closed over $Y$ by $f$.

By \emph{Theorem 5.6} it is seen that $X$ and $f$ have the desired property. This completes the proof.
\end{proof}

\subsection{Existence of \emph{qc} covers}

Now we give the existence of \emph{qc} covers.

\begin{theorem}
Fixed an integral $K$-variety $Y$ and a Galois extension $L$  over  $k(Y)$. Then
there exists an integral $K$-variety $X$ and a surjective morphism $f:X\rightarrow Y$ such that

\begin{itemize}
\item $L=k\left( X\right) $;

\item $f$ is affine;

\item $X$ is a qc over $Y$ by $f$;

\item $X$ is essentially affine in $L^{al}$.
\end{itemize}
\end{theorem}

Such an integral $K$-variety $X$ with a morphism $f$, denoted by $(X,f)$, is said to be a \textbf{\emph{qc} cover} of $Y$.

\begin{proof}
It is immediate from \emph{Lemmas 6.1-2}.
\end{proof}

\begin{lemma}
Let $X$, $Y$ and $Z$ be integral $K$-varieties such that $X$ and $Z$ are {qc} over $Y$. Then $X$ and $Z$ are essentially equal if $k(X)=k(Z)$ and  $X$ and $Z$ are isomorphic.
\end{lemma}

\begin{proof}
It is immediate from \emph{Definition 3.9}.
\end{proof}

\begin{remark}
Let $X$ and $Z$ both be  \emph{qc} over an integral $K$-variety $Y$. Suppose $k(X)=k(Z)$.
In general, it is not true that  $X$ and $Z$  are essentially equal.
\end{remark}

\section{Main Property of \emph{qc} Schemes}

Let $L/K$ be a field extension. The integral schemes $X/Y$ are said to be a \textbf{geometric
model} for the extension $L/K$ if there is a group isomorphism ${Aut}\left(
X/Y\right)\cong Gal\left( L/K\right) $ (e.g., see \cite{SGA1,sv1,sv2}).

In this section we will prove that \emph{qc} schemes afford such a geometric model for an extension of a function field.

\subsection{Function fields of \emph{qc} schemes}

The function fields of \emph{qc} schemes are quasi-galois.

\begin{lemma}
Let $X$ and $Y$ be two integral schemes such that $X$ is qc over $Y$ by
a surjective morphism $f$ of finite type. Then the function field $k\left( X\right) $ is canonically quasi-galois  over the function field $f(Y)$.
\end{lemma}

\begin{proof}
For brevity, assume that $H\triangleq k(Y)$ is contained in $L\triangleq k(X)$. Let $M$ be a conjugation of $L$ over $H$.

Fixed an element $w\in M\setminus H$. There is an element $u\in L\setminus H$ and an $H-$isomorphism $\sigma:L\to M$ such that $w=\sigma(u)$.

As $X$ is essentially affine in an algebraic closure $\Omega$ of $L$, we must have some affine open set $U$ in $X$ such that $u$ is contained in the ring $$A\triangleq\mathcal{O}_{X}(U)\subseteq\Omega$$
where $U$ is contained in $f^{-1}(V)$ for some affine open set $V$ in $Y$.

Put $B=\sigma(A)$. The ring $B$ is a conjugation of $A$ (canonically) over the ring $\mathcal{O}_{Y}(V)$. By \emph{Theorem 5.6} it is seen that $B$ must be contained in $L$ and then the element $w\in B$ is contained in $L$. Hence, we have $M\subseteq L$.

From \emph{Lemma 4.3} it is seen that $L$ is quasi-galois over $H$.
\end{proof}

\subsection{\emph{qc} schemes as geometric models}

Let $X$ and $Y$ be two integral $K$-varieties and let $f:X\rightarrow Y$ be a surjective morphism.

\begin{lemma}
Suppose that $X$ is qc over $Y$ by $f$ and $k\left( X\right) $
is canonically  Galois  over $k\left( Y\right) $. Then there  is a group isomorphism
\begin{equation*}
Aut\left( X/Y\right) \cong Gal\left( k\left( X\right) /k\left( Y\right)
\right) .
\end{equation*}
\end{lemma}

\begin{proof}
In the following we will use the trick developed in \cite{An2,An2*} to prove the second statement above.

In fact, define a mapping
\begin{equation*}
t:Aut\left( X/Y\right) \longrightarrow Gal\left( k\left( X\right) /k\left(
Y\right) \right)
\end{equation*}%
by
\begin{equation*}
\sigma =(\sigma ,\sigma ^{\sharp })\longmapsto t(\sigma )=\left\langle
\sigma ,\sigma ^{\sharp -1}\right\rangle
\end{equation*}%
where $\left\langle \sigma ,\sigma ^{\sharp -1}\right\rangle $ is the map of
$k(X)$ into $k(X)$ given by
\begin{equation*}
\left( U,f\right) \in \mathcal{O}_{X}(U)\subseteq k\left( X\right)
\longmapsto \left( \sigma \left( U\right) ,\sigma ^{\sharp -1}\left(
f\right) \right) \in \mathcal{O}_{X}(\sigma (U))\subseteq k\left( X\right)
\end{equation*}%
for any open set $U$ in $X$ and any element $f\in \mathcal{O}_{X}(U)$.
Here the function field $k(X)$ is taken canonically as the set of elements of the form ${\left( U,f\right)}$.

It is easily seen that $t$ is well-defined.
We will proceed in several steps to prove that $t$ is a group
isomorphism.

\emph{Step 1.} Prove that ${t}$ is injective. Fixed any $\sigma ,\sigma
^{\prime }\in {Aut}\left( X/Y\right) $ such that $t\left( \sigma \right)
=t\left( \sigma ^{\prime }\right) .$ We have
\begin{equation*}
\left( \sigma \left( U\right) ,\sigma ^{\sharp -1}\left( f\right) \right)
=\left( \sigma ^{\prime }\left( U\right) ,\sigma ^{\prime \sharp -1}\left(
f\right) \right)
\end{equation*}%
for any $\left( U,f\right) \in k\left( X\right) .$ In particular, for any $%
f\in \mathcal{O}_{X}(U_{0})$ we have
\begin{equation*}
\left( \sigma \left( U_{0}\right) ,\sigma ^{\sharp -1}\left( f\right)
\right) =\left( \sigma ^{\prime }\left( U_{0}\right) ,\sigma ^{\prime \sharp
-1}\left( f\right) \right)
\end{equation*}%
where $U_{0}$ is an affine open subset of $X$ such that $\sigma \left(
U_{0}\right) $ and $\sigma ^{\prime }\left( U_{0}\right) $ are both
contained in $\sigma \left( U\right) \cap \sigma ^{\prime }\left( U\right) $.

It is seen that
\begin{equation*}
\sigma |_{U_{0}}=\sigma ^{\prime }|_{U_{0}}
\end{equation*}%
holds as isomorphisms of schemes. As $U_{0}$ is dense in $X$, we have
\begin{equation*}
\sigma =\sigma |_{\overline{U_{0}}}=\sigma ^{\prime }|_{\overline{U_{0}}%
}=\sigma ^{\prime }
\end{equation*}%
on the whole of $X$; then
\begin{equation*}
\sigma \left( U\right) =\sigma ^{\prime }\left( U\right) ;
\end{equation*}
hence,
\begin{equation*}
\sigma =\sigma ^{\prime }.
\end{equation*}
This proves that $t$ is an injection.

\emph{Step 2.} Prove that ${t}$ is surjective. Fixed any element $\rho $ of
the group $Gal\left( k\left( X\right) /k\left( Y\right) \right) $.

As $k(X)=\{(U_{f},f):f\in \mathcal{O}_{X}(U_{f})\text{ and }U_{f}\subseteq X%
\text{ is open}\}$, we have
\begin{equation*}
\rho :\left( U_{f},f\right) \in k\left( X\right) \longmapsto \left( U_{\rho
\left( f\right) },\rho \left( f\right) \right) \in k\left( X\right) ,
\end{equation*}
where $U_{f}$ and $U_{\rho (f)}$ are open sets in $X$, $f$ is contained in $%
\mathcal{O}_{X}(U_{f})$, and $\rho (f)$ is contained in $\mathcal{O}%
_{X}(U_{\rho (f)})$.

It is seen  that each element of $Gal\left( k\left( X\right) /k\left(
Y\right) \right) $ gives a unique element of ${Aut}(X/Y)$.

 In fact, fixed any affine open set $V$ of $Y$. It is easily seen that for each affine
open set $U\subseteq \phi ^{-1}(V)$ there is an affine open set $U_{\rho }$
in $X$ such that $\rho $ determines an isomorphism $\lambda _{U}$ between
affine schemes $(U,\mathcal{O}_{X}|_{U})$ and $(U_{\rho },\mathcal{O}%
_{X}|_{U_{\rho }})$. Then
\begin{equation*}
\lambda _{U}|_{U\cap U^{\prime }}=\lambda _{U^{\prime }}|_{U\cap U^{\prime }}
\end{equation*}%
holds as morphisms of schemes for any affine open sets $U,U^{\prime
}\subseteq \phi ^{-1}(V)$.

Glue $\lambda _{U}$ along all such affine open subsets $U\subseteq \phi ^{-1}(V)$, where $V$ runs through all affine open sets in $Y$. Then we have an
automorphism $\lambda $ of the scheme $X$ such that $\lambda |_{U}=\lambda
_{U}$ for any affine open set $U$ in $X$.

It is clear that $t\left(
\lambda \right) =\rho $. Hence, ${t}$ is a surjection.

This completes the
proof.
\end{proof}

\begin{lemma}
Suppose that $X$ is qc over $Y$ by $f$ and $k\left( X\right) $
is canonically  Galois  over $k\left( Y\right) $. Then $f$ is an affine morphism and there is a natural
isomorphism
$$
\mathcal{O}_{Y}\cong f _{\ast }(\mathcal{O}_{X})^{{Aut}\left( X/Y\right) }
$$
where $(\mathcal{O}_{X})^{{Aut}\left( X/Y\right) }(U)$ denotes the invariant
subring of $\mathcal{O}_{X}(U)$ under the natural action of ${Aut}\left(
X/Y\right) $ for any open subset $U$ of $X$.
\end{lemma}

\begin{proof}

Let $\Omega$ be an algebraic closure of the function field $k(X)$. By \emph{Lemma 6.1} assume that $Y$ is essentially affine in $\Omega$ without loss of generality. In particular, suppose $k(Y)\subseteq k(X)$ for brevity.

It is clear that $X$ is essentially affine in $\Omega$. Put $$G={{Aut}\left( X/Y\right) }.$$

Fixed a point $x\in X$ and an affine open set $U$ in $X$ with $x\in U$. Then there must be $${\mathcal{O}_{X}(U)}^{G}=\mathcal{O}_{Y}(V)$$ for  an affine open set $V$ in $Y$ such that $f(x)\in V$ and $U\subseteq f^{-1}(V)$.

Otherwise, if there is some $$w\in {\mathcal{O}_{X}(U)}^{G}\setminus \mathcal{O}_{Y}(V),$$ we have $$w\not \in k(Y),\, w^{-1} \in k(Y)$$ or $$w \in k(Y),\, w^{-1}\not \in k(Y).$$

If $w \in k(Y)$ and $ w^{-1}\not \in k(Y)$, it is seen that $w\not \in k(Y)$ holds since we have $$\sigma(w\cdot w^{-1})=\sigma(w)\cdot \sigma(w^{-1})=1$$ for any $\sigma \in G$.

Hence, for both cases we will have $w\not \in k(Y)$, which will be in contradiction to the fact that $$k(X)^{G}=k(Y)$$ holds by \emph{Lemma 7.2}.

Now consider any open set $U$ in $X$. We have $${\mathcal{O}_{X}(U)}^{G}=\mathcal{O}_{Y}(V)$$ for  an open set $V$ in $Y$ such that $U\subseteq f^{-1}(V)$ since $\mathcal{O}_{X}(U)$ can be regarded as a subring of $\mathcal{O}_{X}(U_{0})$ for an affine open set $U_{0}\subseteq U$. This prove that $$
\mathcal{O}_{Y}=(\mathcal{O}_{X})^{G}
$$ holds.

Conversely, take any affine open set $V$ of $Y$. As $f$ is surjective, it is seen that there is an affine open $U\subseteq f^{-1}(V)$ of $X$ such that $${\mathcal{O}_{X}(U)}^{G}\supseteq\mathcal{O}_{Y}(V).$$ Repeating the same procedure above, we can choose $U$ to be such that $${\mathcal{O}_{X}(U)}^{G}=\mathcal{O}_{Y}(V).$$ It follows that $$U= f^{-1}(V)$$ holds. This proves that $f$ is affine.
\end{proof}

\subsection{Main property of \emph{qc} schemes}

The \emph{qc} schemes behave like quasi-galois extensions of fields.

\begin{theorem}
Let $X$ and $Y$ be two algebraic $K$-varieties such that $k(X)$ is separably generated over $k(Y)$ canonically. Suppose that $X$ is $qc$ over $Y
$ by a surjective morphism $f$ of finite type. Then there are the following statements:

\begin{itemize}
\item $f$ is affine.

\item $k\left( X\right) $ is Galois over $k(Y)$ canonically.

\item There is a group isomorphism
\begin{equation*}
{Aut}\left( X/Y\right) \cong Gal(k\left( X\right) /k(Y)).
\end{equation*}
\end{itemize}

In particular, let $\dim X=\dim Y$. Then $X$ is a pseudo-galois cover of \, $Y$ in the sense of Suslin-Voevodsky.
\end{theorem}

\begin{proof}
It is immediate from \emph{Theorem 4.1} and \emph{Lemmas 7.1-3}.
\end{proof}

Here for pseudo-galois cover, see \cite{sv1,sv2} for definition and property.

\section{\emph{sp}-Completion}

By the graph functor, there is an \emph{sp-}completion of a given integral scheme, which is an integral scheme of the same length but have the maximal combinatorial graph.

The \emph{sp-}completion can give a type of completions of rational maps between schemes.

In the present paper, the \emph{sp-}completion will be applied to definitions of formally unramified extension of fields in \S \emph{9} and of monodromy actions of automorphism groups in \S \emph{12}.

\subsection{The graph functor $\Gamma$ from schemes to graphs}

For convenience, in this subsection we will review the graph functor developed in \cite{An1,An8}. See \cite{An1,An8} for proofs of the results listed below.

Let $X$ be a scheme.
Take any points $x,y$ in $X$.

If $y$ is in the (topological) closure $\overline{\{x\}}$,
$y$ is a \textbf{\ specialization} of $x$ (or,
$x$ is  a \textbf{generalization} of $y $) in $X$, denoted by $x\rightarrow y$.

Put $Sp\left( x\right) =\{y\in X\mid x\rightarrow y\}$. Then
 $Sp\left( x\right) =\overline{\{x\}}$ is an irreducible
closed  subset in $X$.

If $x\rightarrow y$ and $y\rightarrow x$ both hold in $X$, $y$ is
a \textbf{generic specialization} of $x$ in $X$, denoted by $x\leftrightarrow y$.

$x$ is said to be \textbf{generic} (or
\textbf{initial}) in $X$ if we must have $x\leftrightarrow z$ for any $z\in X$
such that $z\rightarrow x$.

$x$ is said to be \textbf{closed} (or
\textbf{final}) if we must have $x\leftrightarrow z$ for any $z\in X$ such that $x\rightarrow z$.

We have the following preliminary facts for specializations.

\begin{lemma}
\emph{(\cite{An1,An8})} For any points $x,y \in X$, we have $x\leftrightarrow y$
in $X$ if and only if $x=y$.
\end{lemma}

\begin{lemma}
\emph{(\cite{An1,An8})} Fixed any specialization $x\rightarrow y$ in $X$. Then
there is an affine open subset $U$ of $X$ such that the two points $x$ and $y
$ are both contained in $U$.
In particular, any affine open set in $X$ containing (the specialization) $y$
must contain (the generalization) $x$.
\end{lemma}

\begin{lemma}
\emph{(\cite{An1,An8})} Any morphism between schemes is
specialization-preserving. That is, fixed any morphism $f:X\to Y$ between schemes. Then there is a specialization $f\left( x\right)
\rightarrow f\left( y\right) $ in $Y$ for any specialization $x\rightarrow y$
in $X$.
\end{lemma}

Let's recall preliminary definitions on combinatorial graphs (see \cite{tut}). Fixed a graph $G$.
Let $V(G)$ be the \textbf{set of vertices} in $G$ and $E(G)$ the \textbf{set of edges} in $G$.

A graph  $H$ is a \textbf{subgraph} of $G$, denoted by $H \subseteq G$, provided that
$V(G)\supseteq V(H)$,
$E(G)\supseteq E(H)$, and
every $L\in E(H)$ has the same ends in $H$ as in $G$.

Now we obtain the graph functor $\Gamma$ from \emph{Lemmas 8.1-3}.

\begin{lemma}
\emph{(\textbf{Graph Functor} \cite{An1,An8})} There exists a covariant functor $\Gamma$, called the \textbf{graph functor}, from the category $Sch$ of schemes to the category $Grph$ of combinatorial graphs,
given in such a natural manner:
\begin{itemize}
\item For a scheme $X$, $\Gamma(X)$ is the graph  in which the vertex
set is the set of points in the underlying space $X$ and the edge set is the
set of specializations in $X$.
Here, for any points $x,y \in X$, we say that there is an edge from $x$ to $y
$ if and only if there is a specialization $x\rightarrow y$ in $X$.

\item For a morphism $f:X\rightarrow Y$ of schemes,
$\Gamma(f): \Gamma(X)\rightarrow \Gamma(Y)$ is the homomorphism between graphs.
Here, any specialization $x \rightarrow y$ in the scheme $X$ as an edge in $\Gamma(X)$,
is mapped by $\Gamma(f)$ into the specialization $f(x)\rightarrow f(y)$ as an edge in $\Gamma(Y)$.
\end{itemize}
\end{lemma}

\subsection{\emph{sp}-completion}

Let $G$ and $H$ be combinatorial graphs.
Recall that an \textbf{isomorphism} $t$ from $G$ onto $H$ is a ordered pair $(t_{V},t_{E})$ satisfying the conditions:
\begin{itemize}
\item $t_{V}$ is a bijection from $V(G)$ onto $V(H)$;

\item $t_{E}$ is a bijection from $E(G)$ onto $E(H)$;

\item Let $x \in V(G)$ and $L\in E(G)$. Then $x $ is incident with $L$ if
and only if $t_{V}(x) \in V(H)$ is incident with $t_{E}(L)\in E(H)$.
\end{itemize}

The following definition says that the graph of a given integral scheme is maximal relative to its function field.

\begin{definition}
An integral scheme $X$ is said to be \textbf{$sp$-complete} if $X$ must be essentially equal to $Y$ for any integral scheme $Y$ such that

\begin{itemize}
\item $\Gamma (X)$ is isomorphic to a subgraph of $\Gamma (Y)$;

\item $k(Y)$ is contained in a separable closure of $k(X)$.
\end{itemize}
\end{definition}

\begin{remark}
The function field of an $sp$-complete integral scheme must be a separable closure. In other words, there is no other separably closed points that can be added to an $sp$-complete integral scheme.
\end{remark}

Now we give the existence of the $sp$-completion of an integral scheme.

\begin{theorem}
For any integral scheme $X$, there is
an integral scheme $X_{sp}$ and a surjective morphism $\lambda_{X}:X_{sp}\to
X$ satisfying the following properties:
\begin{itemize}
\item $\lambda_{X}$ is affine;

\item $X_{sp}$ is $sp$-complete;

\item $X_{sp}$ is essentially affine in $k(X)^{al}$;

\item $k(X_{sp})$ is a separable closure of $k(X)$;

\item $X_{sp}$ is quasi-galois closed over $X$ by $\lambda_{X}$.
\end{itemize}
\end{theorem}

Such a scheme $X_{sp}$ with a morphism $\lambda_{X}$, denoted by $(X_{sp},\lambda_{X})$, is said to be an \textbf{$sp$-completion}
of $X$. We will denote by $Sp[X]$ the set of all \emph{sp}-completions of an integral scheme $X$.

\begin{proof}
(\textbf{Universal Construction for $sp$-Completion}) Here repeat the construction developed in \cite{An8}.

Let $K=k\left( X\right) $ and $L=K^{sep}$.
Fixed a transcendental basis $\Delta_{1}$ of $L$ over $K$ and  a linear basis $\Delta_{2}$ of $L$ as a vector space over $K(\Delta_{1})$. Put $$G=Gal\left( L/K\right) ; \, \Delta=\Delta_{1}\bigcup \Delta_{2}.$$

By
\emph{Lemma 6.1}, without loss of generality, assume that $X$ has a reduced
affine covering $\mathcal{C}_{X}$ with values in $L$. We choose $\mathcal{C}_{X}$ to be maximal (by set inclusion).

We  proceed in several steps such as the following to give the construction:

\begin{itemize}
\item Fixed a local chart $\left( V,\psi _{V},B_{V}\right) \in \mathcal{C}_{X}$. Define $A_{V}=B_{V}\left[ \Delta _{V}\right] $, where $\Delta
_{V}=\{\sigma \left( x\right) \in L:\sigma \in G,x\in \Delta \}. $ Set $i_{V}:B_{V}\rightarrow A_{V}$ to be the inclusion.

\item Let
\begin{equation*}
\Sigma =\coprod\limits_{\left( V,\psi _{V},B_{V}\right) \in \mathcal{C}
_{X}}Spec\left( A_{V}\right)
\end{equation*}
be the disjoint union. Denote by $\pi _{X}:\Sigma \rightarrow X$  the
projection induced by the inclusions $i_{V}$.

\item Given an equivalence relation $R_{\Sigma }$ in $\Sigma $ in such a
manner:
\begin{quotation}
\emph{For any $x_{1},x_{2}\in \Sigma $, we say $x_{1}\sim x_{2}$ if and only if
$j_{x_{1}}=j_{x_{2}}$ holds in $L$.
Here $j_{x}$ denotes the corresponding prime ideal of $A_{V}$ to a point $x\in Spec\left( A_{V}\right) $.}
\end{quotation}
Let $X_{sp}$ be the quotient space $\Sigma /\sim $ and let $\pi_{sp}:\Sigma
\rightarrow X_{sp}$ be the projection of spaces.

\item Set a map $\lambda_{X}:X_{sp}\rightarrow X$ of topological spaces by $\pi_{sp}\left( z\right) \longmapsto \pi _{X}\left( z\right) $ for each $z\in
\Sigma $.

\item Suppose
\begin{equation*}
\mathcal{C}_{X_{sp}}=\{\left( U_{V},\varphi _{V},A_{V}\right) \}_{\left(
V,\psi _{V},B_{V}\right) \in \mathcal{C}_{X}}.
\end{equation*}%
Here $U_{V}=\pi _{X}^{-1}\left( V\right) $ and $\varphi
_{V}:U_{V}\rightarrow Spec(A_{V})$ is the identity map for each $\left(
V,\psi _{V},B_{V}\right) \in \mathcal{C}_{X}$.

\item There is a scheme, namely $X_{sp}$, by gluing the affine schemes $Spec\left( A_{V}\right) $ for all $\left( U_{V},\varphi _{V},A_{V}\right)
\in \mathcal{C}_{X}$ with respect to the equivalence relation $R_{\Sigma }$.
Naturally, $\lambda_{X}$ becomes a morphism of schemes.
\end{itemize}

It is seen that $X_{sp}$ and $\lambda_{X}$ are the desired scheme and morphism, respectively.
\end{proof}

There is the uniqueness of $sp$-completions such as the following.

\begin{lemma}
Fixed any integral $K$-varieties $X$. Then
all sp-completions of $X$ are essentially equal.
\end{lemma}

\begin{proof}
It is seen from \emph{Definition 8.5}, \emph{Remark 8.6}, and \emph{Theorem 8.7}.
\end{proof}

\begin{lemma}
Fixed any two integral $K$-varieties $X$ and $Y$. Suppose that $k(X)$ and $k(Y)$ have the same separable closure. Then either $$Sp[X]=Sp[Y]$$  or $$Sp[X]\bigcap Sp[Y]=\emptyset$$ holds.
\end{lemma}

\begin{proof}
It is immediate  from \emph{Lemma 8.8}.
\end{proof}

\begin{remark}
An \emph{sp}-completion of an integral scheme is {\emph{sp}-complete}. By \emph{sp}-completion we can give a completion of rational maps between integral schemes.
\end{remark}

\begin{remark}
An integral scheme $X$ and its \emph{sp}-completion $X_{sp}$ have the same dimension. However, the \emph{sp}-completion is very complicated and exotic. In general, it is not true that $X_{sp}$ is of finite type over $X$.
 For example, let $t$ be a variable over $\mathbb{Q}$. It is seen that $Spec(\overline{\mathbb{Q}})$ and $Spec(\overline{\mathbb{Q}(t)})$ are \emph{sp}-completions of $Spec(\mathbb{Q})$ and $Spec(\mathbb{Q}(t))$, respectively. Their underlying spaces are very different.
\end{remark}

\begin{remark}
By \emph{Theorem 8.7} it is seen that an \emph{sp}-completion of an integral scheme behaves like a separable closure of a field.
\end{remark}

\section{Unramified Extensions of Function Fields}

In this section we will use \emph{sp}-complete schemes to introduce a notion of formally unramified
extensions over function fields and then give several preliminary properties. The formally unramified extensions will be applied to the computation of \'{e}tale fundamental groups.

\subsection{Basic lemma}

Fixed a field $K$. We have the following basic result.

\begin{lemma}
Let $X$ and $Y$ be two integral $K$-varieties satisfying the two conditions:
\begin{itemize}
\item $Sp[X]=Sp[Y]$ are equal sets.
\item $Aut(X_{sp}/X)\cong Aut(Y_{sp}/Y)$ are isomorphic groups.
\end{itemize}
Then $X$ and $Y$ are isomorphic schemes.
\end{lemma}

\begin{proof}
 Fixed any \emph{sp}-completions $(X_{sp},\lambda_{X})$ of $X$ and $(Y_{sp},\lambda_{Y})$ of $Y$, respectively. As $Sp[X]=Sp[Y]$, we have an isomorphism $$t:X_{sp}\to Y_{sp}.$$

Let $$\sigma:Aut(X_{sp}/X)\to Aut(Y_{sp}/Y)$$ be an isomorphism between groups.

 Take any point $x_{0}\in X$. From \emph{Lemmas 7.2-3} we have $$\lambda^{-1}_{X}(x_{0})=\{g(x_{0})\in X_{sp}:g\in Aut(X_{sp}/X)\};$$
$$\lambda^{-1}_{Y}(t(x_{0}))=\{h(t(x_{0}))\in Y_{sp}:h\in Aut(Y_{sp}/Y)\}.$$
It follows that there exists a morphism $f_{sp}:X_{sp}\to Y_{sp}$ given by $$ g(x_{0})\mapsto \sigma(g)(t(x_{0})).$$

 From $f_{sp}$ we obtain a morphism $f:X\to Y$ given by $$x_{0}=\lambda_{X}(g(x_{0}))\mapsto \lambda_{Y}(\sigma(g)(t(x_{0})))$$ satisfying the property $$f\circ \lambda_{X}=\lambda_{Y}\circ f_{sp}.$$

It is easily seen that $f_{sp}:X_{sp}\to Y_{sp}$ is an isomorphism. Hence, $f:X\to Y$ is an isomorphism.
\end{proof}

\subsection{Formally unramified extensions}

Fixed an integral $K$-variety $X$ over a field $K$. Let $L_{1}$ and $ L_{2}$ be two algebraic extensions over the function field $k(X)$, respectively.

\begin{definition}
$L_{2}$ is said to be a \textbf{finite $X$-formally
unramified Galois extension}  over $L_{1}$ if there are two integral $K$-varieties
$X_{1}$ and $X_{2}$ and a surjective morphism $f:X_{2}\rightarrow X_{1}$
such that
\begin{itemize}
\item $Sp[X]=Sp[X_{1}]=Sp[X_{2}]$;

\item $k\left( X_{1}\right) =L_{1}, \, k\left( X_{2}\right) =L_{2}$;

\item $X_{2}$ is a finite \'{e}tale Galois cover of $X_{1}$ by $f$.
\end{itemize}
In such a case, $X_{2}/X_{1}$ are said to be a \textbf{$X$-geometric model} of the field extension $L_{2}/L_{1}$.
\end{definition}

\begin{remark}
It is seen that such geometric models are unique up to isomorphisms. It follows that the formally unramified extension above is well-defined. In fact, fixed any two geometric models $X_{2}/X_{1}$ and $Y_{2}/Y_{1}$ for the extension $L_{2}/L_{1}$, respectively. By \emph{Lemma 9.1} we must have isomorphisms $X_{2}\cong Y_{2}$ and $X_{1}\cong Y_{1}$, respectively. This is due to the preliminary fact that we have $$Aut(X_{sp}/X_{i})\cong Gal(k(X)^{sep}/L_{i})\cong Aut(X_{sp}/Y_{i})$$ for $i=1,2$.
\end{remark}

\begin{remark}
Suppose that $L_{3}/L_{2}$ and $L_{2}/L_{1}$ both are $X$-formally unramified extensions. Then $L_{3}/L_{1}$ must be $X$-formally unramified.
\end{remark}

\begin{remark}
Note that even for the case that $L_{1}$ and $L_{2}$ are both algebraic extensions of $K$, in general, the
formally unramified  defined in \emph{Definition 9.2} does not
coincide with  unramified that is defined in algebraic number theory.
\end{remark}

\begin{remark}
Note that we define another unramified extensions in \cite{An4*,An5,An8} for
arithmetic schemes, which is a generalization of unramified extensions in
algebraic number theory and hence is different from the above one defined in
\emph{Definition 9.2}.
\end{remark}

\begin{definition}
Let $X$ be an integral $K$-variety  over a field $K$.
 Set
\begin{quotation}
$k(X)^{au}\triangleq $ the smallest field containing all finite $X$-formally unramified
subextensions over $L$ contained in $L^{al}$.
\end{quotation}
The field $k(X)^{au}$ is said to be the \textbf{maximal formally unramified extension} of the function field $k(X)$.
\end{definition}

\begin{lemma}
Let $X$ be an integral $K$-variety and let $L\subseteq k(K)^{au}$ be a finite Galois extension of $k(X)$. Then there are the following statements.

$(i)$ $k(X)^{au}$ is an algebraic Galois extension of  $k(X)$. In particular, $k(X)^{au}$ is a subfield of $k(X)^{sep}$.

$(ii)$ There is a finite $X$-formally unramified Galois extension $M$ of $k(X)$ such that $M\supseteq L$.

$(iii)$ Let $M$ be a finite $X$-formally unramified Galois extension of $k(X)$ such that $M\supseteq L$. Then so is $M$ over $L$.

$(iv)$ $L$ is a finite $X$-formally unramified Galois extension of $k(X)$.
\end{lemma}

\begin{proof}
$(i)$ It is immediate from preliminary facts on field theory.

$(ii)$ It is clear from the assumption that $L\subseteq k(K)^{au}$ holds.

$(iii)$ Take a geometric model $X_{M}/X$ for the extension $M/k(X)$. It reduces to the case that $X_{M}$ and $X$ are both affine schemes.

Suppose $$X=Spec(K_{0}), \, K_{0}=K[t_{1},t_{2},\cdots, t_{n}];$$
$$L=Fr(L_{0}), \, L_{0}=K[t_{1},t_{2},\cdots, t_{n},s_{1},s_{2},\cdots, s_{l}];$$
$$X_{M}=Spec(M_{0}), \, M_{0}=K[t_{1},t_{2},\cdots, t_{n},s_{1},s_{2},\cdots, s_{l},s_{l+1},\cdots, s_{m}].$$
Here, the elements $$t_{1},t_{2},\cdots, t_{n},s_{1},s_{2},\cdots, s_{l},s_{l+1},\cdots, s_{m}$$ are all contained in $k(X)^{au}$, and $$s_{1},s_{2},\cdots, s_{l}$$ and $$s_{1},s_{2},\cdots, s_{l},s_{l+1},\cdots, s_{m}$$ are both supposed to contain all conjugates over the field $k(X)$.

It is easily seen that $X_{M}$ is a finite \'{e}tale Galois cover of $Spec(L_{0})$ from base change of \'{e}tale morphisms.

$(iv)$ It reduces to consider affine schemes. Take $K_{0},L_{0},M_{0}$ as in $(iii)$ above.

Let $\mathfrak{P}$ be a maximal ideal of $M_{0}$. Just check what one has done in algebraic number theory. Put $$\mathfrak{p}=\mathfrak{P}\bigcap K_{0}.$$
Then $\mathfrak{p}$ is a maximal ideal of $K_{0}$.

Conversely, let $\mathfrak{p}$ be a maximal ideal of $K_{0}$. From definition for \'{e}tale morphisms, we have one and only one maximal ideal $\mathfrak{P}$ of $M_{0}$ that is over the maximal ideal $\mathfrak{p}$.

Likewise, there is one and only one maximal ideal $\mathfrak{P}_{0}$ of $L_{0}$ such that $\mathfrak{P}|\mathfrak{P}_{0}$ and $\mathfrak{P}_{0}|\mathfrak{p}$ hold.

It follows that $Spec(L_{0})$ must be unramified over $X$ and hence \'{e}tale over $X$.
This completes the proof.
\end{proof}

\subsection{Arithmetic unramified extension}

There is another type of unramified extensions, the arithmetic unramified extensions over the ring $\mathcal{O}_{K}$ of algebraic integers  of a number field $K$.

\textbf{Convention.} In this subsection, an \textbf{integral $\mathbb{Z}$-variety} is defined to be an integral scheme surjectively over $Spec(\mathbb{Z})$; an \textbf{arithmetic variety} is an integral scheme surjectively over $Spec(\mathbb{Z})$ of finite type.

Likewise, we have the following basic lemma.

\begin{lemma}
Let $X$ and $Y$ be two integral $\mathbb{Z}$-varieties satisfying the two conditions:
\begin{itemize}
\item $Sp[X]=Sp[Y]$ are equal sets.
\item $Aut(X_{sp}/X)\cong Aut(Y_{sp}/Y)$ are isomorphic groups.
\end{itemize}
Then $X$ and $Y$ are isomorphic schemes.
\end{lemma}

Fixed an integral $\mathbb{Z}$-variety $X$ over a field $K$. Let $L_{1}$ and $ L_{2}$ be two algebraic extensions over the function field $k(X)$, respectively.

\begin{definition}
The field $L_{2}$ is said to be a \textbf{finite $X$-unramified Galois extension}  over $L_{1}$ if there are two integral $\mathbb{Z}$-varieties
$X_{1}$ and $X_{2}$ and a surjective morphism $f:X_{2}\rightarrow X_{1}$
such that
\begin{itemize}
\item $Sp[X]=Sp[X_{1}]=Sp[X_{2}]$;

\item $k\left( X_{1}\right) =L_{1}, \, k\left( X_{2}\right) =L_{2}$;

\item $X_{2}$ is a finite \'{e}tale Galois cover of $X_{1}$ by $f$.
\end{itemize}
In such a case, $X_{2}/X_{1}$ are said to be a \textbf{$X$-geometric model} of the field extension $L_{2}/L_{1}$.
\end{definition}

\begin{remark}
Let $L_{1}\subseteq L_{2} \subseteq L_{3}$ be function fields over a number
field $K$. Suppose that $L_{2}/L_{1}$ and $L_{3}/L_{2}$ are $X$-unramified
extensions. Then $L_{3}$ is $X$-unramified over $L_{1}$.
\end{remark}

\begin{definition}
Let $X$ be an integral $\mathbb{Z}$-variety.
 Set
\begin{quotation}
$k(X)^{un}\triangleq $ the smallest field containing all finite $X$-unramified
subextensions over $L$ contained in $L^{al}$.
\end{quotation}
The field $k(X)^{un}$ is said to be the \textbf{maximal unramified extension} of the function field $k(X)$.
\end{definition}

\begin{lemma}
Let $X$ be an integral $\mathbb{Z}$-variety and let $L\subseteq k(K)^{un}$ be a finite Galois extension of $k(X)$. Then there are the following statements.

$(i)$ $k(X)^{un}$ is an algebraic Galois extension of  $k(X)$.

$(ii)$ There is a finite $X$-unramified Galois extension $M$ of $k(X)$ such that $M\supseteq L$.

$(iii)$ Let $M$ be a finite $X$-unramified Galois extension of $k(X)$ such that $M\supseteq L$. Then so is $M$ over $L$.

$(iv)$ $L$ is a finite $X$-unramified Galois extension of $k(X)$.
\end{lemma}

\begin{proof}
Repeat what we have done in proving \emph{Lemma 9.8}.
\end{proof}

\begin{remark}
It is seen that for the case of an algebraic extension, the unramified
extension defined in \emph{Definition 9.12} coincides exactly with that in
algebraic number theory.
\end{remark}

It appears that unramified extensions for an arithmetic variety and for an algebraic $K$-variety have some common properties.
However, they are very different. For example, we have $$k(Spec(\mathbb{Z}))=\mathbb{Q}^{un}; \, k(Spec(\mathbb{Q}))=\mathbb{Q}^{sep}=\overline{\mathbb{Q}}. $$ In particular, for arithmetic varieties, we have a  stronger result such as the following.

\begin{theorem}
Let $X$ and $Y$ be two arithmetic varieties such that $k(X)=k(Y)$. Then $Sp[X]=Sp[Y]$ holds, i.e., $X$ and $Y$ have the same {sp}-completions; moreover, $X$ and $Y$ are isomorphic.
\end{theorem}

\begin{proof}
By \emph{Lemma 9.9} it suffices to prove that $X$ and $Y$ have a common \emph{sp}-completion.
Fixed any \emph{sp}-completions $(X_{sp},\lambda_{X})$ of $X$ and $(Y_{sp},\lambda_{Y})$ of $Y$, respectively. It reduces to prove that $X_{sp}$ and $Y_{sp}$ are isomorphic.

In the following we will proceed in several steps to prove that there exists an isomorphism $$f_{sp}:X_{sp}\to Y_{sp}.$$

\emph{Step 1}. We have $\Omega\triangleq k(X)^{al}=k(Y)^{al}$. As $X_{sp}$ and $Y_{sp}$ are both \emph{qc} over $Spec(\mathbb{Z})$, it is seen that for any affine open set $U$ in $X$ there must be an affine open set $V$ in $Y$ such that $$f_{U}:U\to V$$ is an isomorphism of schemes which is induced from an isomorphism $$\sigma_{U}:\mathcal{O}_{Y}(V) \to \mathcal{O}_{X}(U)$$ between subrings of $\Omega$.

By the universal construction in \S \emph{8.2}, we have a morphism $$f_{sp}:X_{sp}\to Y_{sp}$$ such that there is the restriction $${{f_{sp}}|}_{U}=f_{U}$$ to each affine open set $U$ in $X$.

\emph{Step 2}. By \emph{Theorems 5.6,8.7}, assume that $\sigma_{U}$ is an identity map without loss of generality. It is seen that $f_{sp}$ is an injective morphism.

\emph{Step 3}. Take any closed point $y_{0}$ in $Y_{sp}$. Let $B(y_{0})\subseteq \Omega$ be a subring such that the affine open set $V(y_{0})=Spec(B(y_{0}))$ in $Y$ containing the point $y_{0}$. Denote by $j_{y_{0}}$ the prime ideal in $B(y_{0})$ corresponding to  $y_{0}$.

By \emph{Theorem 5.6} it is seen that $B(y_{0})$ is a ring over $\mathbb{Z}$ generated by the set
$$\{t_{1},\cdots,t_{n}\}\bigcup \Delta$$ where $t_{1},\cdots,t_{n}$ are variables over $\mathbb{Q}$, $\dim X=n$, and $$\Delta=\overline{\mathbb{Q}(t_{1},\cdots,t_{n})}\setminus \mathbb{Z}.$$ This is due to the fact that $Y_{sp}$ is \emph{qc} over $Spec(\mathbb{Z})$.

It is seen that such a prime ideal $j_{y_{0}}$ contains a unique prime $\mathfrak{P}\in \mathcal{O}_{K}$ over a prime $\mathfrak{p}\in \mathbb{N}$, where $\mathcal{O}_{K}$ is the ring of the algebraic integers of a number field $K$ and  $j_{y_{0}}$ is a maximal ideal of  the ring $B(y_{0})$ generated by a set $\Delta_{\mathfrak{P}}$ containing the subset $$\{\mathfrak{P}\}\bigcup \{t_{1},\cdots,t_{n}\}$$ of $\Delta$.

As $X_{sp}$ is \emph{qc} over $Spec(\mathbb{Z})$, we have a point $x_{0}\in X$ such that $$j_{x_{0}}=j_{y_{0}}.$$ Then we have $f_{sp}(x_{0})=y_{0}$.

This completes the proof.
\end{proof}

\begin{remark}
By\emph{Theorem 9.15} it is seen that  \emph{unramified extension}, the notion for arithmetic varieties  given in \cite{An4*,An5,An8}, as in \emph{Definition 9.12}, are  well-defined.
\end{remark}

\section{Algebraic Fundamental Groups}

In this section we will give the computation of algebraic fundamental groups.

\subsection{A universal cover for an \'{e}tale fundamental group}

For an integral $K$-variety $X$, let $k(X)^{au}$ denote the maximal formally unramified extension of the function field $k(X)$.

\begin{lemma}
For any integral $K$-variety $X$, there exists an
integral $K$-variety $X_{{et}}$ and a surjective morphism $
p_{X}:X_{et}\rightarrow X$ satisfying the properties:

\begin{itemize}
\item $p_{X}$ is affine;

\item $k\left( X_{et}\right) ={k(X)}^{au}$;

\item $X_{et}$ is qc over $X$ by $p_{X}$;

\item $k\left( X_{et}\right) $ is Galois over $k\left( X\right) ;$

\item $X_{et}$ is essentially affine in ${k(X)}^{au}$.
\end{itemize}
\end{lemma}

Such an integral $K$-variety $X_{et}$ with a morphism $p_{X}$, denoted by $\left( X_{et},p_{X}\right) $, is called a \textbf{universal
cover} over $X$ for the \'{e}tale fundamental group $\pi _{1}^{et}\left(
X\right) $.

\begin{proof}
(\textbf{Universal Construction for the Cover})
By \emph{Lemma 6.1}, without loss of generality,
assume that $X$ has a reduced
affine covering $\mathcal{C}_{X}$ with values in $k(X)^{al}.$ Let $\mathcal{C}_{X}$ be maximal by set inclusion.

We will proceed in several steps:

\begin{itemize}
\item Fixed a set $\Delta $ of generators of the field ${k(X)}^{au}$ over $k(X)$.

\item For any local chart $\left( V,\psi _{V},B_{V}\right) \in \mathcal{C}_{X}$,
define $A_{V}=B_{V}\left[ \Delta _{V}\right] $, that is, $A_{V}$ over $B_{V}$ generated by the set $$\Delta _{V}=\{\sigma \left( x\right)
\in {k(X)}^{au}:\sigma \in Gal({k(X)}^{au}/k(X)),x\in \Delta \}. $$

Let $i_{V}:B_{V}\rightarrow A_{V}$
be the inclusion.

\item Assume that
\begin{equation*}
\Sigma =\coprod\limits_{\left( V,\psi _{V},B_{V}\right) \in \mathcal{C}
_{X}}Spec\left( A_{V}\right)
\end{equation*}
is the disjoint union. Let $\pi _{X}:\Sigma \rightarrow X$ be the
projection induced by the inclusions $i_{V}$.

\item Define an equivalence relation $R_{\Sigma }$ in $\Sigma $ in such a
manner:
\begin{quotation}
\emph{For any $x_{1},x_{2}\in \Sigma $, we say $x_{1}\sim x_{2}$ if and only if
$j_{x_{1}}=j_{x_{2}}$ holds in $L$, where $j_{x}$ denotes the corresponding prime ideal of $A_{V}$ to a point $x$ in $Spec\left( A_{V}\right) $.}
\end{quotation}

Let $X_{et}$ be the quotient space $\Sigma /\sim $ and let $
\pi_{{et}}:\Sigma \rightarrow X_{et}$ be the projection of
spaces.

\item Set a map $p_{X}:X_{{et}}\rightarrow X$ of spaces by $
\pi_{{et}}\left( z\right) \longmapsto \pi _{X}\left( z\right) $ for
each $z\in \Sigma $.

\item Suppose
\begin{equation*}
\mathcal{C}_{X_{{et}}}=\{\left( U_{V},\varphi _{V},A_{V}\right)
\}_{\left( V,\psi _{V},B_{V}\right) \in \mathcal{C}_{X}}.
\end{equation*}
Here $U_{V}=\pi _{X}^{-1}\left( V\right) $ and $\varphi
_{V}:U_{V}\rightarrow Spec(A_{V})$ is the identity map for each $\left(
V,\psi _{V},B_{V}\right) \in \mathcal{C}_{X}$.

\item There is a scheme, namely $X_{{et}}$, obtained by gluing the
affine schemes $Spec\left( A_{V}\right) $ for all $\left( U_{V},\varphi
_{V},A_{V}\right) \in \mathcal{C}_{X}$ with respect to the equivalence
relation $R_{\Sigma }$. Naturally, $p_{X}$ becomes a
morphism of schemes.
\end{itemize}

It is seen that $X_{{et}}$ and $p_{X}$ satisfy the properties.
This completes the proof.
\end{proof}

\subsection{A computation of \'{e}tale fundamental groups}

By \emph{Lemma 10.1} we have the following result.

\begin{theorem}
For any integral $K$-variety $X$, there exists a group isomorphism
\begin{equation*}
\pi _{1}^{et}\left( X\right) \cong Gal\left( {k(X)}^{au}/k\left( X\right)
\right).
\end{equation*}
\end{theorem}

\begin{proof}
Assume that $X$
has a reduced a reduced affine covering $\mathcal{C}_{X}$ with values in $k(X)^{al}$ without loss of generality.

Let $\Delta \subseteq k(X)^{au}\setminus
k\left( X\right) $ be a set of generators of the field $k(X)^{au}$ over $k(X) $. Put $$I=\{\text{finite subsets of }\Delta \}.$$

We will proceed in several steps to give the proof.

\emph{Step 1.} Fixed any ${\alpha }$ in $I$. Repeating the universal construction in \S \emph{10.1}
for $\alpha $, i.e., replacing $\Delta$ by $\alpha$, we have an integral $K$-variety $X_{\alpha
}$ and a surjective morphism $f_{\alpha }:X_{\alpha }\rightarrow X$
satisfying the properties:
\begin{itemize}
\item $f_{\alpha }$ is affine;

\item $k\left( X_{\alpha }\right) \subseteq k\left( X\right) ^{au}$;

\item $X_{\alpha }$ is \emph{qc} over $X$ by $f_{\alpha }$;

\item $k\left( X_{\alpha }\right) $ is Galois over $k\left( X\right)$;

\item $X_{\alpha}$ is essentially affine in $k(X)^{au}$.
\end{itemize}

\emph{Step 2.} Let $\alpha \subseteq \beta$ be in $I$.
 By \emph{Step 1} we have integral $K$-varieties $X_{\alpha }$
and $X_{\beta }$ which are \emph{qc} over $X$, respectively. There is a surjective morphism $f_{\alpha }^{\beta }:X_{\beta } \to X_{\alpha }$
satisfying the properties:
\begin{itemize}
\item $f_{\alpha }^{\beta }$ is affine;

\item $f_{\beta }=f_{\alpha }\circ f_{\alpha }^{\beta }$;

\item $X_{\beta}$ is \emph{qc} over $X_{\alpha}$ by $f_{\alpha }^{\beta }$;

\item $k\left( X_{\beta }\right) $ is Galois over $k\left( X_{\alpha}\right)$.
\end{itemize}
Here
$f_{\alpha }^{\beta }$ is obtained in a canonical manner similar to $f_{\alpha }$.

It is clear that there is a $\gamma $ in $I$ such that $\gamma \supseteq
\alpha $ and $\gamma \supseteq \beta .$ Hence, we have an integral $K$-variety $X_{\gamma}$ that is \emph{qc} over $X_{\beta }$ and over $X_{\alpha }$, respectively.

\emph{Step 3.} For any $\alpha ,\beta $ in $I$, we say $\alpha \leq \beta $
if and only if $\alpha \subseteq \beta .$ Then $I$ is a partially
ordered set.

Hence, $\{k\left( X_{\alpha }\right) ;i_{\alpha }^{\beta
}\}_{\alpha \in I}$ is a direct system of groups, where each $$i_{\alpha
}^{\beta }:k\left( X_{\alpha }\right) \rightarrow k\left( X_{{\beta }
}\right) $$ is a homomorphism of fields canonically induced by $f_{\alpha }^{\beta }$.

Let $(X_{et},p_{X})$ be a universal cover for $\pi_{1}^{et}(X)$.
For the fields, we have $$k\left( X_{et}\right)=k(X)^{au} ={\lim_{\longrightarrow}}_{\alpha \in I}
k\left( X_{\alpha }\right) .$$
For the Galois groups, we have
$$Gal\left( k\left( X_{et}\right) /k\left( X\right) \right) \cong
{\lim_{\longleftarrow}} _{\alpha \in I}Gal\left( k\left( X_{\alpha }\right)
/k\left( X\right) \right) .$$

\emph{Step 4.} Let $$[X]_{au} =\{X_{\alpha }:\alpha \in
I\}.$$

Then $[X]_{au}$ is a directed set. Here for any $X_{\alpha },X_{\beta }\in [X]_{au}$, we say $$X_{\alpha
}\leq X_{\beta }$$ if and only if $X_{\beta }$ is \emph{qc}
over $X_{\alpha }$.

Fixed a geometric point $s$ of $X$ over $k(X)^{al}.$ Put
$$
[X]_{et}=
\{\text{finite \'{e}tale Galois covers of } X \text{
over }s\}.
$$

Then $[X]_{et}$ is a directed set. Here for
any $X_{1},X_{2}\in [X]_{et}$, we say $$X_{1}\leq X_{2}$$
if and only if $X_{2}$ is a finite \'{e}tale Galois cover over $X_{1}.$

\emph{Step 5.} Fixed any $X_{\alpha },X_{\beta }\in [X]_{au}$.

It is seen that $X_{\alpha }$ and $X_{\beta}$ both are finite \'{e}tale Galois covers of $X$ by \emph{Lemma 9.8}.

Let $X_{\beta }$ be
\emph{qc} over $X_{\alpha }$. Then $X_{\beta }$ is a \'{e}tale finite Galois cover of
$X_{\alpha }$ from \emph{Lemma 9.8} again.

Hence,  $[X]_{au}$ is a directed subset of $[X]_{et}$.

\emph{Step 6.} Let
$Z\in [X]_{et}$. We have $k\left( Z\right) \subseteq k\left( X\right) ^{au}.$ It is seen that $k\left(
Z\right) $ is a finite unramified Galois extension of $k\left( X\right) $.

Let $\alpha \subseteq k\left( Z\right) \setminus k\left( X\right) $ be a set
of generators of the field $k\left( Z\right) $ over $k\left( X\right) .$ As $\alpha \in I$ is finite and $\Delta $ is infinite, there is a finite set $\beta \in I$
such that $$\alpha \subsetneqq \beta \subsetneqq \Delta .$$

We have $$X_{\beta }\in [X]_{au}$$ such that $X_{\beta }$
is \emph{qc} over $Z.$

Hence, $[X]_{au}$ is a co-final directed subset in $X_{et} $.

\emph{Step 7.} Now by \emph{Steps 1-6} above we have
\begin{equation*}
\begin{array}{l}
\pi _{1}^{et}\left( X\right)\\

 ={\lim_{\leftarrow}} _{Z\in [X]_{et}}Aut\left( Z/X\right) \\

\cong {\lim_{\leftarrow }}_{Z\in [X]_{au}}Aut\left(
Z/X\right) \\

\cong {\lim_{\leftarrow }}_{Z\in [X]_{au}}Gal\left(
k\left( Z\right) /k\left( X\right) \right) \\

={\lim_{\leftarrow}} _{\alpha \in I}Gal\left( k\left( X_{\alpha }\right)
/k\left( X\right) \right) \\

\cong Gal\left( k\left( X_{et}\right) /k\left( X\right) \right) \\

=Gal\left( k\left( X\right) ^{au}/k\left( X\right) \right) .
\end{array}
\end{equation*}

This completes the proof.
\end{proof}

\begin{remark}
Let $X$ be an integral $K$-variety. From \emph{Theorem 10.2} we have $$\pi _{1}^{et}\left( X_{et}\right)\cong \{0\}.$$ For example, let $X=Spec(\mathbb{Q})$. We have $X_{et}=Spec(\overline{\mathbb{Q}})$ and hence $$\pi _{1}^{et}\left( Spec(\overline{\mathbb{Q}})\right)\cong \{0\}.$$
\end{remark}

\subsection{A prior estimate of the \'{e}tale fundamental group}

In this subsection we will introduce  the \emph{qc} fundamental group of an algebraic variety.  We will prove that the \'{e}tale fundamental group is a normal subgroup of the \emph{qc} fundamental group.

Fixed an algebraic $K$-variety $X$. Let $\Omega$ be the separable closure of a separably generated extension  of the function field $k\left( X\right) $.

Define $$[X;\Omega]_{qc}$$ to be the set of algebraic $K$-varieties $Z$ satisfying the conditions:
\begin{itemize}
\item $k(Z)$ is contained in $\Omega $;

\item There is a
surjective morphism $f:Z\rightarrow X$ of finite type such that $Z$ is
\emph{qc} over $X.$
\end{itemize}

In \cite{An6}, we require such an additional condition that
\begin{quotation}
\emph{$Z$ has a reduced affine covering with values in $\Omega $}.
\end{quotation}
However, from \emph{Lemma 6.1} it is seen that there is no essential difference between the two conditions.

There are preliminary facts on the set $[X;\Omega]_{qc}$ such as the following.

\begin{lemma}
For any $Z_{1},Z_{2}\in [X;\Omega]_{qc}$, there is a third $Z_{3}\in [X;\Omega]_{qc}$ such that $Z_{3}$ is qc over $Z_{1}$ and $Z_{2},$ respectively.
\end{lemma}

\begin{lemma}
Let $Z_{1},Z_{2},Z_{3}\in [X;\Omega]_{qc}$. Suppose that $Z_{2}$ is qc over $Z_{1}$ and $Z_{3}$ is qc over
$Z_{2}$. Then $Z_{3}$ is qc over $Z_{1}$.
\end{lemma}

Here, \emph{Lemmas 10.4-5} above can be proved in a manner similar to what we have done for the proof of \emph{Theorem 10.2}.

Set a partial order $\leq$ in the set $[X;\Omega]_{qc}$ in such
a manner:
\begin{quotation}
\emph{For any $Z_{1},Z_{2}\in [X;\Omega]_{qc}$, we say
$$
Z_{1}\leq Z_{2}
$$
if and only if there is a surjective morphism $\varphi :Z_{2}\rightarrow Z_{1}$ of
finite type such that $Z_{2}$ is qc over $Z_{1}.$}
\end{quotation}

By \emph{Lemmas 10.4-5} it is seen that $[X;\Omega]_{qc}$ is a directed set and
$$
\{Aut\left( Z/X\right) :Z\in [X;\Omega]_{qc} \}
$$
is an inverse system of groups.

Now we introduce the following definition.

\begin{definition}
Let $X$ be an algebraic $K$-variety. Suppose that $\Omega$ is the separable closure of a separably generated extension of $k\left( X\right) $. The
inverse limit
$$
\pi _{1}^{qc}\left( X;\Omega \right) \triangleq {\lim_{\longleftarrow}}
_{Z\in [X;\Omega]_{qc}}{Aut\left( Z/X\right)}
$$
of the inverse system $\{Aut\left( Z/X\right) :Z\in [X;\Omega]_{qc} \}$ of groups is said to be the \textbf{\emph{qc} fundamental group} of  $X$ with coefficients in $\Omega $.
\end{definition}

We have the following result on the \emph{qc} fundamental group, which will be prove in \S \emph{10.5}.

\begin{theorem}
Let $X$ be an algebraic $K$-variety. Suppose that $\Omega$ is the separable closure of a separably generated extension of $k(X)$. There are the
following statements.

$\left( i\right) $ There is a group isomorphism
\begin{equation*}
\pi _{1}^{qc}\left( X;\Omega \right) \cong Gal\left( {\Omega }/k\left(
X\right) \right) .
\end{equation*}

$\left( ii\right) $ There is a group isomorphism
$$
\pi _{1}^{et}\left( X;s\right) \cong \pi _{1}^{qc}\left( X;\Omega \right)
_{et}
$$ for a geometric point $s$ of $X$ over $\Omega $,
where $\pi _{1}^{qc}\left( X;\Omega \right) _{et}$ is a normal subgroup of $\pi
_{1}^{qc}\left( X;\Omega \right) $.
\end{theorem}

\begin{remark}
Let $X$ be an algebraic $K$-variety. Define
$$
\pi _{1}^{qc}\left( X \right)\triangleq \pi _{1}^{qc}\left( X;{k(X)}^{sep} \right).
$$
It is seen that there is a group isomorphism
$$
\pi _{1}^{qc}\left( X \right) \cong Gal\left( {k(X)}^{sep}/k\left( X\right)
\right).
$$
\end{remark}

\begin{remark}
Let $X$ be an algebraic $K$-variety. The quotient group
$$
\pi _{1}^{br}\left( X\right) =\frac {\pi _{1}^{qc}\left( X;k\left( X\right)
^{sep}\right)}{\pi _{1}^{qc}\left( X;k\left( X\right) ^{sep}\right) _{et}}
$$
is said to be the \textbf{ramified group} (or \textbf{branched group}) of  $X$.  By \emph{Lemma 9.8} it is seen that $k(X)^{sep}$ is a Galois extension of $k(X)^{au}$; from \emph{Theorems 10.2,10.7} we have $$\pi _{1}^{br}\left( X\right)= \frac{ Gal(k(X)^{sep}/k(X))}{Gal(k(X)^{au}/k(X))}.$$ The ramified group $\pi _{1}^{br}\left( X\right)$  reflects the
topological properties of the scheme $X$ such as the branched covers. In deed, such ramified groups will play an important role in giving the anabelian functors in the present paper.
\end{remark}

\subsection{A universal cover for the \emph{qc} fundamental group}

Let $X$ be an algebraic $K$-variety. Suppose that  $\Omega$ is the separable closure of a separably generated extension of the function field $k\left( X\right) $.

Assume that $X$ has a reduced affine covering with
values in $\Omega^{al}$ without
loss of generality. Let
$$
G=Gal\left( \Omega /k\left( X\right) \right)
$$
and let
$$
\Delta \subseteq \Omega \setminus k\left( X\right)
$$ be a set
of generators of $\Omega $ over $k\left( X\right) .$ By \emph{Theorem 4.1} it is seen that $\Omega$ is Galois over $k(X)$.

Repeating the universal construction for \'{e}tale fundamental group
in \S \emph{10.1}, we have an integral variety $X_{\Omega }$ and a morphism $f_{\Omega }$ such as in the
following lemma.

\begin{lemma}
There is an integral
$K$-variety $X_{\Omega }$ and a surjective morphism $f_{\Omega }:X_{\Omega
}\rightarrow X$ satisfying the conditions:

\begin{itemize}
\item $k\left( X_{\Omega }\right) =\Omega $;

\item $f_{\Omega }$ is affine;

\item $X_{\Omega }$ is qc over $X$ by $f_{\Omega }$;

\item $k\left( X_{\Omega }\right) $ is Galois over $k\left( X\right) ;$

\item $X_{\Omega }$ is essentially affine in $\Omega$.
\end{itemize}
\end{lemma}

Such an integral $K$-variety $X_{\Omega }$ with a morphism $f_{\Omega}$, denoted by $(X_{\Omega },f_{\Omega })$, is said to be  a
\textbf{universal cover} over $X$ for the \emph{qc} fundamental group group $\pi _{1}^{qc}\left(
X;\Omega \right)$.

\subsection{Proof of \emph{Theorem 10.7}}

Now we can prove the main result above on the \emph{qc} fundamental group in \S \emph{10.4}.

\begin{proof}
(\textbf{Proof of \emph{Theorem 10.7}}) We will proceed in several steps.

\emph{Step 1.} By \emph{Theorem 7.4} we have
$$
\begin{array}{l}
Gal\left( \Omega /k\left( X\right) \right) \\
\cong {\lim_{\leftarrow }}_{Z\in [X;\, \Omega]_{qc} }{Gal\left(
k\left( Z\right) /k\left( X\right) \right) } \\
\cong {\lim_{\leftarrow }}_{Z\in [X;\, \Omega]_{qc} }{Aut\left(
Z/X\right) } \\
={\pi }_{1}^{qc}\left( X;\Omega \right)
\end{array}
$$ according to preliminary facts on Galois groups.

\emph{Step 2.} Fixed a geometric point $s$ of $X$ over $\Omega$.
Let $[X;s]_{et}$ be the set of finite \'{e}tale
Galois covers of $X$ over the geometric point $s$. For any
$Z_{1},Z_{2}\in [X;s]_{et}$ we say
$$
Z_{1}\leq Z_{2}
$$
if and only if $Z_{2}$ is a finite \'{e}tale Galois cover of $Z_{1}.$ Then $[X;s]_{et}$ is a partially ordered set. Put
\begin{equation*}
[X;s]_{qc}\triangleq [X;\Omega]_{qc} \bigcap [X;s]_{et}.
\end{equation*}

Let $Z_{1},Z_{2}\in [X;s]_{qc}$. It is easily seen that $Z_{2}$ is a finite \'{e}tale Galois cover of $Z_{1}$ if
and only if $Z_{2}$ is \emph{qc} over $Z_{1}$.

It follows that $[X;s]_{qc} $ is a co-final directed subset in $[X;s]_{et} $.

\emph{Step 3.} Now consider the universal covers $X_{\Omega}$ and $X_{et}$ of $X$ for the groups $\pi_{1}^{qc}(X;\Omega)$ and $\pi_{1}^{et}(X;s)$, respectively.
From \emph{Step 7} in \S \emph{10.2} we have
\begin{equation*}
\begin{array}{l}
Gal\left(k(X_{ et})/k\left( X\right) \right) \\
\cong {\lim_{\leftarrow }}_{Z\in [X;s]_{et} }{Gal\left(
k\left( Z\right) /k\left( X\right) \right) } \\
\cong {\lim_{\leftarrow }}_{Z\in [X;s]_{et}}{Aut\left(
Z/X\right) } \\
\cong {\pi }_{1}^{et}\left( X;s\right) .
\end{array}
\end{equation*}

By \emph{Step 2} we have
\begin{equation*}
\begin{array}{l}
Gal\left(k(X_{ et})/k\left( X\right) \right) \\

\cong {\lim_{\leftarrow }}_{Z\in [X;s]_{et}}{Gal\left( k\left(
Z\right) /k\left( X\right) \right) } \\

\cong {\lim_{\leftarrow }}_{Z\in [X;s]_{qc}  }{
Gal\left( k\left( Z\right) /k\left( X\right) \right) }
\end{array}
\end{equation*}
since $[X;s]_{qc} $ is  co-final  in $[X;s]_{et} $.

On the other hand, we have
$$k(X_{\Omega})={\lim _{\longrightarrow}}_{Z \in [X;\, \Omega]_{qc} }{k(Z)}$$ and
$$k(X_{et})={\lim _{\longrightarrow}}_{Z \in [X;\, s]_{et} }{k(Z)}={\lim _{\longrightarrow}}_{Z \in [X;\, s]_{qc} }{k(Z)}$$
as direct limits of direct systems of groups for the function fields.

It is seen that $${\lim _{\longrightarrow}}_{Z \in [X;\, \Omega]_{qc} }{k(Z)}$$ is an extension of the field  $${\lim _{\longrightarrow}}_{Z \in [X;\, s]_{qc} }{k(Z)} .$$

It follows that $$k(X)^{au}=k(X_{et})$$ is a subfield of $$\Omega=k(X_{\Omega}).$$

Then we have a tower of Galois extensions of function fields $$k(X)\subseteq k(X)^{au}\subseteq \Omega$$ from \emph{Lemma 10.1}.

It is seen that
${\pi }_{1}^{et}\left( X;s\right)$ is isomorphic to the normal subgroup $$Gal\left(k(X)^{au}/k\left(
X\right) \right) $$ of the group $Gal\left( \Omega /k\left( X\right) \right) $.
Hence, ${\pi }_{1}^{et}\left( X;s\right)$ is isomorphic to a normal subgroup of ${\pi }_{1}^{qc}\left( X;\Omega \right)$
 since by \emph{Step 1} we have $$Gal\left( \Omega /k\left( X\right) \right) \cong {\pi }_{1}^{qc}\left( X;\Omega \right) .$$
This completes the proof.
\end{proof}

\section{Monodromy Actions}

 Naturally there exist three types of monodromy actions for a given integral $K$-variety, as we have done for arithmetic varieties in \cite{An8}:
\begin{itemize}
\item Monodromy action of \'{e}tale fundamental group on the universal cover;

\item Monodromy action of absolute Galois group on the \emph{sp}-completion;

\item Monodromy action of ramified group on the \emph{sp}-completion.
\end{itemize}

\subsection{Monodromy actions of \'{e}tale fundamental groups}

From \emph{Lemma 10.1} and \emph{Theorem 10.2} it is seen there is the below preliminary fact on
 the \'{e}tale fundamental group of an
integral variety.

\begin{lemma}
For an integral $K$-variety $X$, there is an isomorphism
\begin{equation*}
Aut\left( X_{et}/X\right)  \cong \pi _{1}^{et}\left( X\right)
\end{equation*}
between groups, where $\left( X_{et},p_{X}\right)$ is a universal cover of $X$ for the
\'{e}tale fundamental group $\pi_{1}^{et}(X)$.
\end{lemma}

Now let $X$ and $Y$ be two integral $K$-varieties. Suppose that $\left( X_{et},p_{X}\right)$ and $\left( Y_{et},p_{Y}\right)$ are
universal covers of $X$ and $Y$ for the \'{e}tale fundamental groups $\pi_{1}^{et}(X)$ and $\pi_{1}^{et}(Y)$,
respectively.

By \emph{Lemma 11.1} it is seen that each group homomorphism
\begin{equation*}
\sigma :\pi _{1}^{et}\left( X\right) \rightarrow \pi _{1}^{et}\left(
Y\right) .
\end{equation*}
gives a group homomorphism, namely
\begin{equation*}
\sigma :Aut\left( X_{et}/X\right) \rightarrow Aut\left( Y_{et}/Y\right) .
\end{equation*}
The converse is true.

Here is the monodromy action of \'{e}tale fundamental groups on the universal covers.

\begin{lemma}
Assume that there is a group homomorphism
\begin{equation*}
\sigma :Aut\left( X_{et}/X\right) \rightarrow Aut\left( Y_{et}/Y\right) .
\end{equation*}
Then there is a bijection
\begin{equation*}
\tau :Hom\left( X,Y\right) \to Hom\left( X_{et},Y_{et}\right) , f\mapsto f_{et}
\end{equation*}
between sets given in a canonical manner:

\begin{itemize}
\item Let $f\in Hom\left( X,Y\right)$. Then the map
\begin{equation*}
g\left( x_{0}\right) \longmapsto \sigma \left( g\right) \left( f\left(
x_{0}\right) \right)
\end{equation*}
defines a morphism
\begin{equation*}
f_{et}:X_{et}\rightarrow Y_{et}
\end{equation*}
for any $x_{0}\in X$ and any $g\in Aut\left( X_{et}/X\right) .$

\item Let $f_{et}\in Hom\left( X_{et},Y_{et}\right)$.
Then the map
\begin{equation*}
p_{X}\left( x\right) \longmapsto p_{Y}\left( f_{et}\left( x\right) \right)
\end{equation*}
defines a morphism
\begin{equation*}
f:X\rightarrow Y
\end{equation*}
for any $x\in X_{et}.$
\end{itemize}

In particular, we have
\begin{equation*}
f\circ p_{X}=p_{Y}\circ f_{et}.
\end{equation*}
\end{lemma}

\begin{proof}
It is immediate from \emph{Lemma 10.1} and \emph{Theorem 10.2}.
\end{proof}

\subsection{Monodromy actions of absolute Galois groups}

Let $X$ and $Y$ be two integral $K$-varieties. Consider the $sp$-completions $(X_{sp},\lambda_{X})$ and $(Y_{sp},\lambda_{Y})$ and the universal covers $\left( X_{et},p_{X}\right)$ and $\left( Y_{et},p_{Y}\right)$ for  $\pi_{1}^{et}(X)$ and $\pi_{1}^{et}(Y)$,
respectively.

 It is easily seen that there are isomorphisms
$$
Gal(k(X)^{sep}/k(X))\cong Aut(X_{sp}/X);
$$
$$
Gal(k(Y)^{sep}/k(Y))\cong Aut(Y_{sp}/Y).
$$
between groups from \emph{Lemma 7.2} and \emph{Theorem 8.7}.

Here is the monodromy action of absolute Galois groups on the \emph{sp}-completions.

\begin{lemma}
Suppose that there is a group
homomorphism $$\sigma: Aut(X_{sp}/X) \to Aut(Y_{sp}/Y).$$ Then there is a bijection
\begin{equation*}
\tau :Hom\left( X,Y\right) \to Hom\left( X_{sp},Y_{sp}\right) , f\mapsto
f_{sp}
\end{equation*}
between sets given in a canonical manner:

\begin{itemize}
\item Let $f\in Hom\left( X,Y\right)$. Then the map
\begin{equation*}
g\left( x_{0}\right) \longmapsto \sigma \left( g\right) \left( f\left(
x_{0}\right) \right)
\end{equation*}
defines a morphism
\begin{equation*}
f_{sp}:X_{sp}\rightarrow Y_{sp}
\end{equation*}
for any $x_{0}\in X$ and any $g\in Aut(X_{sp}/X)$.

\item Let $f_{sp}\in Hom\left( X_{sp},Y_{sp}\right)$. Then the map
\begin{equation*}
\lambda_{X}\left( x\right) \longmapsto \lambda_{Y}\left( f_{sp}\left(
x\right) \right)
\end{equation*}
defines a morphism
\begin{equation*}
f:X\rightarrow Y
\end{equation*}
for any $x\in X_{sp}.$
\end{itemize}

In particular, we have
\begin{equation*}
f\circ \lambda_{X}=\lambda_{Y}\circ f_{sp}.
\end{equation*}
\end{lemma}

\begin{proof}
It is immediate from \emph{Lemma 7.2} and \emph{Theorem 8.7}.
\end{proof}

\subsection{Monodromy actions of ramified groups}

To start with, let's prove a preparatory lemma.

\begin{lemma}
Let $X$ be an integral $K$-variety. Suppose that $(X_{sp},\lambda_{X})$  is an $sp$-completion of $X$ and $\left( X_{et},p_{X}\right)$ is a universal cover of $ X$ for the \'{e}tale fundamental group $\pi_{1}^{et}(X)$. Then there exists canonically a surjective morphism $q_{X}:X_{sp}\to
X_{et}$ satisfying the below properties:
\begin{itemize}
\item $q_{X}$ is affine;

\item $\lambda_{X}=p_{X}\circ q_{X}$;

\item $X_{sp}$ is qc over $X_{et}$ by $q_{X}$;

\item $(X_{sp},q_{X})$ is an sp-completion of $X_{et}$.
\end{itemize}
In particular, we have $q_{X}=\lambda_{X_{et}}.$
\end{lemma}

\begin{proof}
Repeat the universal construction for an \emph{sp}-completion of the integral scheme $X_{et}$ in \S \emph{8}. Then we have a morphism $$q_{X}=\lambda_{X_{et}}:X_{sp}\to
X_{et}$$ such that $(X_{sp},\lambda_{X_{et}})$ is an \emph{sp}-completion of $X_{et}$.
\end{proof}

Let $X$ and $Y$ be two integral $K$-varieties. Consider the $sp$-completions $(X_{sp},\lambda_{X})$ and $(Y_{sp},\lambda_{Y})$ and the universal covers $\left( X_{et},p_{X}\right)$ and $\left( Y_{et},p_{Y}\right)$ for  $\pi_{1}^{et}(X)$ and $\pi_{1}^{et}(Y)$,
respectively.

From \emph{Remark 10.9} we have the ramified groups $$\pi _{1}^{br}\left( X\right)= \frac{ Gal(k(X)^{sep}/k(X))}{Gal(k(X)^{au}/k(X))};$$
$$\pi _{1}^{br}\left( Y\right)= \frac{ Gal(k(Y)^{sep}/k(Y))}{Gal(k(Y)^{au}/k(Y))}.$$

It follows that we have the following result.

\begin{lemma}
For any integral $K$-varieties $X$ and $Y$, there are group isomorphisms
$$
\pi _{1}^{br}\left( X\right)\cong Gal(k(X)^{sep}/k(X)^{au}) \cong Aut(X_{sp}/X_{et});
$$
$$
\pi _{1}^{br}\left(Y\right)\cong Gal(k(Y)^{sep}/k(Y)^{au})\cong Aut(Y_{sp}/Y_{et}).
$$
\end{lemma}

Here is the monodromy action of ramified groups on the \emph{sp}-completions. It will play an important role in the anabelian geometry.

\begin{lemma}
Suppose that there is a group
homomorphism $$\sigma: Aut(X_{sp}/X_{et}) \to Aut(Y_{sp}/Y_{et}).$$ Then there is a bijection
\begin{equation*}
\tau :Hom\left( X_{et},Y_{et}\right) \to Hom\left( X_{sp},Y_{sp}\right) , f\mapsto
f_{sp}
\end{equation*}
between sets given in a canonical manner:

\begin{itemize}
\item Let $f\in Hom\left( X_{et},Y_{et}\right)$. Then the map
\begin{equation*}
g\left( x_{0}\right) \longmapsto \sigma \left( g\right) \left( f\left(
x_{0}\right) \right)
\end{equation*}
defines a morphism
\begin{equation*}
f_{sp}:X_{sp}\rightarrow Y_{sp}
\end{equation*}
for any $x_{0}\in X_{et}$ and any $g\in Aut(X_{sp}/X_{et})$.

\item Let $f_{sp}\in Hom\left( X_{sp},Y_{sp}\right)$. Then the map
\begin{equation*}
\lambda_{X_{et}}\left( x\right) \longmapsto \lambda_{Y_{et}}\left( f_{sp}\left(
x\right) \right)
\end{equation*}
defines a morphism
\begin{equation*}
f:X_{et}\rightarrow Y_{et}
\end{equation*}
for any $x\in X_{sp}.$
\end{itemize}

In particular, we have
\begin{equation*}
f\circ \lambda_{X_{et}}=\lambda_{Y_{et}}\circ f_{sp}.
\end{equation*}
\end{lemma}

\begin{proof}
It is immediate from \emph{Lemmas 11.3-4}.
\end{proof}

\section{Proof of the Main Theorem}

\subsection{Preliminary lemmas}

Let $X$ and $Y$ be two integral $K$-varieties. Fixed any $sp$-completions $(X_{sp},\lambda_{X})$ and $(Y_{sp},\lambda_{Y})$ of $X$ and $Y$, and any universal covers $\left( X_{et},p_{X}\right)$ and $\left( Y_{et},p_{Y}\right)$ of $X$ and $Y$ for the groups $\pi_{1}^{et}(X)$ and $\pi_{1}^{et}(Y)$,
respectively.

There are several results on the \emph{sp}-completions and the universal covers of $X$ and $Y$, respectively (c.f. \cite{An8}).

\begin{remark}
From a viewpoint of \emph{sp}-completion, it is seen that arithmetic varieties and integral $K$-varieties are very different. In fact, for arithmetic varieties $Z_{1},Z_{2}$, by \emph{Theorem 9.15} we have $$k(Z_{1})=k(Z_{2})\implies Sp[Z_{1}]=Sp[Z_{2}]. $$ However, for integral $K$-varieties $Z_{1},Z_{2}$, from \emph{Lemma 8.9} it is seen that $$k(Z_{1})=k(Z_{2})\implies Sp[Z_{1}]=Sp[Z_{2}]$$ does not hold in general.
\end{remark}

\begin{lemma}
Suppose $k(X)=k(Y)$ and $Sp[X]=Sp[Y]$. Then
there exists a bijection $\tau$ from $Hom(X,Y)$ onto $Hom(X_{sp},Y_{sp})$ given
in a canonical manner.
In particular, $Hom(X,Y)$ must be a non-void set.
\end{lemma}

\begin{proof}
Fixed any \emph{sp}-completions $X_{sp}$ and $Y_{sp}$ of $X$ and $Y$, respectively. By \emph{Lemma 8.8} there is an isomorphism $$f_{sp}:X_{sp}\to Y_{sp}$$ according to the assumption that $Sp[X]=Sp[Y]$ holds.

By \emph{Theorem 8.7} we have
\begin{equation*}
Aut(X_{sp}/X)\cong Gal(k(X)^{sep}/k(X)) \cong Aut(Y_{sp}/Y).
\end{equation*}
From \emph{Lemma 11.3}  we immediately obtain the desired properties.
\end{proof}

\begin{remark}
From a viewpoint of graph functor $\Gamma$, it is seen that arithmetic varieties and integral $K$-varieties are also very different. In fact, for arithmetic varieties $Z_{1},Z_{2}$,  we have $$k(Z_{1})\subseteq k(Z_{2})\implies \Gamma(Z_{1})\subseteq \Gamma(Z_{2}). $$ However, for integral $K$-varieties $Z_{1},Z_{2}$, in general, it is not true that $$k(Z_{1})\subseteq k(Z_{2})\implies \Gamma(Z_{1})\subseteq \Gamma(Z_{2}) $$ holds.
\end{remark}

\begin{lemma}
Let $k(X)$ be separably generated over $k(Y)$. Then there are the following statements:
\begin{itemize}
\item There is a homomorphism
$$\sigma_{sp}: Gal(k(X)^{sep}/k(X)) \to Gal(k(Y)^{sep}/k(Y)).$$

\item There is a bijection $\tau$ from $Hom(X,Y)$ onto $Hom(X_{sp},Y_{sp})$ given in a canonical manner. In particular,  $Hom(X,Y)$ is empty  if and only if so is $Hom(X_{sp},Y_{sp})$.

\item Let $\Gamma(X_{sp})\supseteq \Gamma(Y_{sp})$. Then $Hom(X,Y)$ must be a non-void set.
\end{itemize}
\end{lemma}

\begin{proof}
It suffices to prove the third statement. Suppose that $\Gamma(Y_{sp})$ is a subgraph of $\Gamma(X_{sp})$.
Take the \emph{sp}-completions $(X_{sp},\lambda_{X})$ and $(Y_{sp},\lambda_{Y})$
of $X$ and $Y$, respectively.

It is seen that the ring $B$ of an affine open set $V$ in $Y_{sp}$ must
be embedded into the ring $A $ of some certain affine open set $U$ in $X_{sp}$ as a subring.
In deed, choose $A$ to be the ring over $B$ generated by the set
$$\Delta_{B}=\{\sigma(w):w\in \Delta,\sigma\in Gal(k(X)^{sep}/k(Y)^{sep})\}$$
where $\Delta \subseteq k(X)^{sep}\setminus k(Y)^{sep}$ is a set of generators of $k(X)^{sep}$ over $k(Y)^{sep}$. It is easily seen that $U=Spec(A)$ is an affine open set in an \emph{sp}-completion of $X$ that is essentially equal to the given $X_{sp}$ from the universal construction for $X_{sp}$ in \S \emph{8.2}.

Conversely, each $A$ must contain some $B$. In fact, let $\xi$ and $\eta$ be the generic points of $X$ and $Y$, respectively.
Take any point $y_{0}$ in $V$. We have the specializations $$\eta \to y_{0} \text{ in } Y_{sp};$$
$$\xi \to y_{0} \text{ in } X_{sp}.$$

From \emph{Lemma 8.2} we have an affine open set $U=Spec(A)$ in $X_{sp}$ containing $\xi$ and $y_{0}$; then, $U$ also contains $\eta$. Hence,  $A$ contains some $B$ such that $y_{0}\in V=Spec(B)$.

It follows that there is a
homomorphism
\begin{equation*}
f_{U}:U=Spec(A)\to V=Spec(B)
\end{equation*}
defined by the inclusion. This gives us a scheme homomorphism
\begin{equation*}
f_{sp}:X_{sp}\to Y_{sp}.
\end{equation*}

By the projections $\lambda_{X}:X_{sp}\to X$ and $\lambda_{Y}:Y_{sp}\to Y$
we have a unique homomorphism $f:X \to Y$ satisfying the condition
\begin{equation*}
\lambda_{sp}\circ f_{sp}=f \circ \lambda_{sp}.
\end{equation*}
This completes the proof.
\end{proof}

\begin{lemma}
Let $k(X)$ be separably generated over $k(Y)$. Then there are the following statements:
\begin{itemize}
\item There is a homomorphism
$$\sigma_{br}: Gal(k(X)^{sep}/k(X)^{au}) \to Gal(k(Y)^{sep}/k(Y)^{au}).$$

\item There is a bijection $\tau$ from $Hom(X_{et},Y_{et})$ onto $Hom(X_{sp},Y_{sp})$ given in a canonical manner. In particular,  $Hom(X_{et},Y_{et})$ is empty  if and only if so is $Hom(X_{sp},Y_{sp})$.

\item Let $\Gamma(X_{sp})\supseteq \Gamma(Y_{sp})$. Then $Hom(X_{et},Y_{et})$ must be a non-void set.
\end{itemize}
\end{lemma}

\begin{proof}
It is immediate from \emph{Lemmas 11.4-6,12.4}.
\end{proof}

\begin{lemma}
Let $k(X)$ be separably generated over $k(Y)$. Then there are the following statements:
\begin{itemize}
\item There is a homomorphism
$$\sigma_{et}: Gal(k(X)^{au}/k(X)) \to Gal(k(Y)^{au}/k(Y)).$$

\item There is a bijection $\tau$ from $Hom(X,Y)$ onto $Hom(X_{au},Y_{au})$ given in a canonical manner. In particular,  $Hom(X,Y)$ is empty  if and only if so is $Hom(X_{au},Y_{au})$.

\item Let $\Gamma(X_{sp})\supseteq \Gamma(Y_{sp})$. Then $Hom(X_{et},Y_{et})$ must be a non-void set.
\end{itemize}
\end{lemma}

\begin{proof}
It is immediate from \emph{Lemmas 11.2,12.4}.
\end{proof}

\subsection{Proof of the main theorem}

Now we can give the proof of the main theorem in the paper.

\begin{proof}
\textbf{(Proof of Theorem 1.3)} Noticed that from \emph{Lemma 11.5} we have the ramified groups $$\pi _{1}^{br}\left( X\right)= \frac{ Gal(k(X)^{sep}/k(X))}{Gal(k(X)^{au}/k(X))}\cong Aut(X_{sp}/X_{et});$$
$$\pi _{1}^{br}\left( Y\right)= \frac{ Gal(k(Y)^{sep}/k(Y))}{Gal(k(Y)^{au}/k(Y))}\cong Aut(Y_{sp}/Y_{et}).$$

Then we have

$$Hom(X,Y)\cong Hom (\pi _{1}^{br}\left( X\right),\pi _{1}^{br}\left( Y\right))$$ $$\cong Hom(\frac{Aut(X_{sp}/X)}{Aut(X_{et}/X)},\frac{Aut(Y_{sp}/Y)}{Aut(Y_{et}/Y)})$$
from \emph{Lemmas 11.1-4,11.6,12.2,12.4-6}.

This completes the proof.
\end{proof}

\end{document}